\documentclass[11pt]{amsart}
\usepackage[T1]{fontenc}
\usepackage[a4paper,margin=2.5cm]{geometry}
\usepackage{mathtools,amssymb}
\usepackage{amsmath}
\usepackage{mathrsfs}
\usepackage{bm}
\usepackage{soul}
\usepackage{tikz}
\usetikzlibrary{arrows.meta}
\usepackage{epsfig}
\usepackage{enumitem}
\allowdisplaybreaks
\numberwithin{equation}{section}
\usepackage[most]{tcolorbox}
\usepackage{xcolor}

\usepackage{graphicx}
\usepackage{booktabs}
\graphicspath{{figures/}}

\theoremstyle{plain}
\newtheorem{theorem}{Theorem}[section]
\newtheorem{lemma}[theorem]{Lemma}
\newtheorem{proposition}[theorem]{Proposition}
\newtheorem{corollary}[theorem]{Corollary}

\theoremstyle{definition}
\newtheorem{definition}[theorem]{Definition}
\newtheorem{assumption}[theorem]{Assumption}

\theoremstyle{remark}
\newtheorem{remark}[theorem]{Remark}

\newcommand{\R}{\mathbb{R}}

\newcommand{\Om}{\Omega}

\newcommand{\Mplus}{\mathcal{M}^{+}_{\lambda,\Lambda}}
\newcommand{\Mminus}{\mathcal{M}^{-}_{\lambda,\Lambda}}
\newcommand{\Sn}{\mathbb{S}^{N}}
\newcommand{\OmP}{\Omega^{+}}
\newcommand{\OmM}{\Omega^{-}}
\newcommand{\OmZ}{\Omega^{0}}

\newcommand{\norm}[1]{\left\lVert #1\right\rVert}

\newcommand{\pos}[1]{\left(#1\right)_{+}}
\newcommand{\negpart}[1]{\left(#1\right)_{-}}
\newcommand{\Ipar}{\mathcal{I}}
\newcommand{\NtildeP}{\widetilde N_{+}}
\newcommand{\psafe}{p_{\mathrm{S}}^{+}}
\newcommand{\NCL}{\widetilde N_{\mathrm{CL}}}
\newcommand{\pCL}{p_{\mathrm{CL}}^{+}}
\newcommand{\pHS}{p_{\mathrm{HS}}^{+}} 

\usepackage[dvipsnames]{xcolor}
\newif\ifdraftnotes
\draftnotestrue

\usepackage{cite}
\usepackage[hidelinks]{hyperref}
\usepackage[nameinlink,capitalize]{cleveref}

\begin{document}

\title[Weighted Liouville theorems and a priori bounds]
{Weighted Liouville Theorems and A Priori Bounds for Indefinite
Superlinear Equations with Pucci's Maximal Operator}

\author[K. Rathore]{Karan Rathore}
\address{%
Department of Mathematics,
Visvesvaraya National Institute of Technology,
Nagpur 440010, India}
\email{%
ds23mth004@students.vnit.ac.in;
rathorekaran020@gmail.com}

\author[M. Mallick]{Mohan Mallick}
\address{%
Department of Mathematics,
Visvesvaraya National Institute of Technology,
Nagpur 440010, India}
\email{%
mohanmallick@mth.vnit.ac.in;
mohan.math09@gmail.com}
\thanks{Corresponding author: Mohan Mallick.}

\author[R. B. Verma]{Ram Baran Verma}
\address{%
Department of Mathematics,
SRM University AP,
Amaravati, Andhra Pradesh 522502, India}
\email{%
rambaran.v@srmap.edu.in;
rambv88@gmail.com}

\date{\today}


\begin{abstract}
Consider 
\begin{equation*}\
\left\{
\begin{aligned}
-\Mplus(D^{2}u)&=\mu u+a(x)u^{p}\quad\text{in}~\Omega,\\
u&=0\quad\text{on}~~\partial\Omega,
\end{aligned}
\right.
\end{equation*}
where $\Om$ is a smooth bounded domain, $a$ is a continuous function which is allowed to change the sign. Moreover, set of zeros of $a$ may have non-empty interior. Main objective of this article is to establish the uniform bound on the set of positive solutions of above problem. This approach requires a suitable Liouville type theorem. We first prove the Liouville type theorem for $-\Mplus(D^{2}v)=h(x_{1})f(v),$ where  \(h\) is a strictly increasing one-dimensional weight and \(f\) is positive and
nondecreasing.

\end{abstract}

\maketitle
\section{Introduction}\label{sec:introduction}
This article considers the following problem
\begin{equation}\label{eq:problem1}
\begin{cases}
-\Mplus(D^{2}u)=\mu u+a(x)u^{p} & \text{in }\Om,\\
u>0 & \text{in }\Om, \\
u=0 & \text{on }\partial\Om,
\end{cases}
\end{equation}
where \(\mu\in\R\), \(\Om\subset\R^{N}\), \(N\ge 2\), $\Mplus$ is a Pucci's operator and $a$ is a continuous function which is allowed to change the sign. Moreover, the set of zeros can have non-empty interior. As we know the existence of multiple positive solutions is obtained by applying the Degree theory from non-linear analysis. One of the main steps is to establish the uniform bound on the set of positive solutions, see\cite{BerestyckiCapuzzoDolcettaNirenberg1994,AmannLopezGomez1998,Du2004}.Main objective of this article is to establish uniform a prior estimate for positive solutions of \eqref{eq:problem1}. It has been accomplished by following the standard blow up argument introduced in \cite{MR619749}. This argument require a Liouville type theorem for the corresponding blow up equation. In the case where $a$ is positive and domain is convex the corresponding blow up equation is turned out to be 
\[-\Mplus(D^{2}u)=u^{p}~~\text{in}~\mathbb{R}^{N}.\]
For appropriate choice of $p$ Liouville type theorem theorem for this problem has been established in \cite{CutriLeoni2000}. If the domain is not convex and blow up sequence of points converge to boundary then we get similar equation in the half space, see\cite{QuaasSirakov2006}. However, if $a$ changes sign then along with the above result we also need the Liouville type result for 
\[-\Mplus(D^{2}v)=h(x_{1})f(v)\qquad\text{in}~~\R^{N},\]
under suitable assumptions on $h$ and $f.$\\
Liouville type theorem for such problems in the context of semilinear elliptic equation first of all has been established in \cite{BerestyckiCapuzzoDolcettaNirenberg1994,berestycki1995variational}. The proofs in these articles heavily rely on variational techniques. In fact, variational methods exploit the principal spectral structure of the cone through a positive homogeneous harmonic weight. The resulting energy estimates, together with their scaling properties, determine the critical exponent for the Liouville-type nonexistence result. In the context of non-variational form elliptic equations these problems has been studied in \cite{AmannLopezGomez1998}, where authors use the harnack inequality. We allow the zero set of \(a\) to be nontrivial and, under the configurations described in \cref{gem_a}, to contain regions with nonempty interior. Thus the superlinear term acts as a source on
\(
\OmP:=\{x\in\Om:a(x)>0\},
\)
as an absorption term on
\(
\OmM:=\{x\in\Om:a(x)<0\},
\)
and vanishes on
\(
\OmZ:=\{x\in\Om:a(x)=0\}.
\)
The change of sign removes the sign-definite structure available for positive weights and creates additional concentration mechanisms near the interfaces of \(\OmZ\). The resulting question is whether all positive solutions of \cref{eq:problem1} remain uniformly bounded when \(\mu\) ranges over a fixed compact subset of its admissible range.

The paper has two main components. We first establish a nonlinear weighted Liouville theorem for bounded nonnegative viscosity solutions \(v\in C^2(\R^N)\) of
\[
-\Mplus(D^{2}v)=h(x_{1})f(v)\qquad\text{in }\R^{N}.
\]
Under suitable regularity and monotonicity assumptions on \(h\) and \(f\), every nontrivial solution is strictly increasing in the \(x_{1}\)-direction. If, in addition, \(h(t)\to+\infty\) as \(t\to+\infty\), then no nontrivial bounded nonnegative entire solution exists. We then apply this Liouville theorem to obtain a uniform \(L^\infty\) a priori estimate for the positive solutions of \cref{eq:problem1}, under structural assumptions on the positive, negative, and zero regions of \(a\) and power-type descriptions of its behavior near the relevant interfaces. The connection between the two results appears in the blow-up analysis: whereas concentration inside a region on which \(a\) has a fixed sign produces a constant-coefficient limiting equation, concentration near a zero or sign-changing interface preserves the local vanishing profile of \(a\) and produces precisely a weighted entire equation. Excluding this interface limit is the central role of the weighted Liouville theorem.

For semilinear equations, the relation between indefinite problems and
nonlinear Liouville theorems was developed by Berestycki,
Capuzzo-Dolcetta, and Nirenberg
\cite{BerestyckiCapuzzoDolcettaNirenberg1994}; see also the indefinite
Liouville results of Zhu \cite{Zhu1999}. Amann and L\'opez-G\'omez
treated a broad nonvariational class of superlinear indefinite boundary
value problems, allowing the weight to vanish on sets of positive
measure, and combined a priori estimates with fixed-point and
bifurcation methods \cite{AmannLopezGomez1998}. Moving-plane methods
were used by Chen and Li to sharpen estimates for semilinear equations
with indefinite coefficients \cite{ChenLi2003}. The closest antecedent
to the present Liouville problem is the work of Du and Li
\cite{DuLi2005}. They proved nonlinear Liouville theorems for bounded
nonnegative solutions of
\[
 -\Delta v=h(x_{1})f(v)\quad\text{in }\R^{N}
\]
and used them to obtain a priori estimates in cases in which a thick
zero region or a power-type interface profile prevents reduction to a
constant-coefficient equation.

The fundamental connection between a priori bounds and Liouville-type nonexistence theorems on the entire space was pioneered in the semilinear setting by Gidas and Spruck \cite{MR619749} using blow-up analysis. For semilinear elliptic equations with variable coefficients, the relation between indefinite, sign-changing weights and positive solutions was systematically developed on bounded domains by Alama and Tarantello \cite{AlamaTarantello1993}, who established that when the parameter $\lambda$ exceeds the first eigenvalue $\lambda_1(\Omega)$, the negative part of the weight acts as a crucial stabilizing force. This variational and bifurcation lineage was subsequently expanded by Alama and Tarantello \cite{MR1414377} to investigate competing power-type terms of the form $-\Delta u = \lambda u + k(x)u^q - h(x)u^p$ (with $1 < q < p$). To bypass traditional odd-symmetry requirements on the nonlinearity, Alama and Del Pino \cite{AlamaDelPino1996} deployed infinite-dimensional Morse theory and non-symmetric linking techniques, utilizing critical group computations to characterize existence and multiplicity in highly indefinite variational structures \cite{AlamaDelPino1996}.\\
Concurrently, topological and variational methods on bounded domains were refined to establish the exact nodal (sign-changing) properties of solutions and relax global growth restrictions, see Li and Wang \cite{LiWang2000}.\\
The extension of these indefinite semilinear problems to the unbounded domain $\mathbb{R}^N$ introduced severe analytical hurdles due to the loss of standard Sobolev compactness. Costa and Tehrani \cite{CostaTehrani2001} treated the noncoercive entire space problem $-\Delta u - \lambda h(x)u = a(x)g(u)$ in $\mathbb{R}^N$ by requiring the sign-changing weight to be strictly negative at infinity ($\lim_{|x|\to\infty} a(x) < 0$) \cite{CostaTehrani2001}. However, Du \cite{Du2004} identified a fundamental structural barrier in unbounded domains: for $\lambda > 0$, positive solutions do not decay to zero at infinity but instead approach a positive constant, placing them entirely outside the standard variational space $D^{1,2}(\mathbb{R}^N)$ \cite{Du2004}. To bypass this lack of compactness and the fact that $0$ is a non-simple eigenvalue on $\mathbb{R}^N$, Du developed a domain-approximation framework on expanding balls \cite{Du2004}. By obtaining uniform, domain-independent a priori bounds and splitting the approximating global bifurcation branches into minimal and unbounded non-minimal components, Du recovered the complete bounded-domain multiplicity diagram on the entire space $\mathbb{R}^N$ \cite{Du2004}.\\
In the context of critical Sobolev exponent growth, Chang-Shou Lin \cite{Lin1998} established a fundamental Liouville-type nonexistence theorem for the scalar curvature equation $\Delta u + K(x) u^{\frac{n+2}{n-2}} = 0$ in $\mathbb{R}^n$ under a gradient nondegeneracy condition on the homogeneous coefficient $K(x)$ \cite{Lin1998}.  On the other hand, the limits of spectral sufficiency under non-coercive potentials were explored in the Neumann setting by Tehrani \cite{Tehrani1996}.\\
For semilinear equations, the relation between indefinite problems and nonlinear Liouville theorems was further developed by Berestycki, Capuzzo-Dolcetta, and Nirenberg \cite{BerestyckiCapuzzoDolcettaNirenberg1994}; see also the indefinite Liouville results of Zhu \cite{Zhu1999}. Amann and López-Gómez \cite{AmannLopezGomez1998} treated a broad nonvariational class of superlinear indefinite boundary value problems, allowing the weight to vanish on sets of positive measure, and combined a priori estimates with fixed-point and bifurcation methods \cite{AmannLopezGomez1998}. Moving-plane methods were used by Chen and Li \cite{ChenLi2003} to sharpen estimates for semilinear equations with indefinite coefficients. The closest antecedent to the present Liouville problem is the work of Du and Li \cite{DuLi2005}. They proved nonlinear Liouville theorems for bounded nonnegative solutions of
\[
-\Delta v = h(x_1)f(v) \quad \text{in } \mathbb{R}^N
\]
and used them to obtain a priori estimates in cases in which a thick zero region or a power-type interface profile prevents reduction to a constant-coefficient equation. \\

Barrios,
Del Pezzo, Garc\'ia-Meli\'an, and Quaas obtained a corresponding
indefinite Liouville theorem for the fractional Laplacian and applied it
to a priori bounds and existence
\cite{BarriosDelPezzoGarciaMelianQuaas2017}.
In another direction, the global weak-Harnack and quantitative Hopf
method of Sirakov \cite{Sirakov2020}, combined with comparison ideas
from the indefinite theory, led L\'opez-G\'omez and Sampedro to broad
semilinear a priori bounds which, within their setting and for
\(N\ge3\), do not require a prescribed decay rate of the potential near
its zero set \cite{LopezGomezSampedro2025}. Consequently, the
power-type interface assumption used here with a direct description of the 
structural hypothesis adapted to our fully nonlinear blow-up argument,
not as a generally necessary condition for indefinite elliptic
problems.

The passage from the Laplacian to \(\Mplus\) is not formal. The natural
framework is that of viscosity solutions
\cite{CrandallIshiiLions1992,CaffarelliCabre1995}; there is in general no
variational formulation, and the difference of two solutions does not
satisfy a linear equation. Instead, it is controlled by inequalities
involving both Pucci extremal operators. Moving-plane arguments for
fully nonlinear equations on unbounded domains were developed by Li
\cite{Li1991}, while the spectral theory needed for nonproper Dirichlet
problems involves two principal half-eigenvalues rather than a single
principal eigenvalue \cite{QuaasSirakov2008}.

Liouville theory for fully nonlinear uniformly elliptic equations also
has a substantial independent development. Cutr\`i and Leoni proved
whole-space Liouville properties for nonnegative viscosity
supersolutions \cite{CutriLeoni2000}. Felmer and Quaas analyzed radial
nonnegative entire solutions of Lane--Emden equations driven by Pucci's
extremal operators and identified operator-dependent critical exponents
\cite{FelmerQuaas2003}. Quaas and Sirakov established a half-space
Liouville theorem for the Pucci Lane--Emden equation and used it in the
study of nonproper superlinear Dirichlet problems
\cite{QuaasSirakov2006}; related half-space inequalities were studied by
Leoni \cite{Leoni2012}. These results provide the constant-coefficient
whole-space and boundary inputs for a blow-up argument, but they do not
by themselves yield the weighted equation generated at a sign/zero
interface,
\begin{equation}\label{eq:intro-weighted-limit}
 -\Mplus(D^{2}v)=h(x_{1})f(v)\quad\text{in }\R^{N}.
\end{equation}
The simultaneous presence of the fully nonlinear operator and the one-directional weight in \cref{eq:intro-weighted-limit} is the main contribution of this article.

Our first result is a weighted Liouville theorem for bounded
nonnegative solutions of \cref{eq:intro-weighted-limit}. Roughly, the
weight \(h\) is nondecreasing, vanishes at the origin, and satisfies a
quantitative compatibility condition there; it may change sign, and the
one-sided profile \(h(t)=t_{+}^{\gamma}\) is also included. The
nonlinearity \(f\) is positive and nondecreasing on \((0,\infty)\), with
the mild behavior at zero specified in Assumption \ref{ass:liouville-data}. We
prove that every nontrivial bounded nonnegative solution is strictly
increasing in the \(x_{1}\)-direction. Moreover, if 
\(h(t)\to+\infty\) as \(t\to+\infty\), then no such nontrivial solution
exists; see \cref{thm:weighted-liouville}.

The proof adapts the moving-plane method
\cite{GidasNiNirenberg1979,Li1991}. Its fully nonlinear feature is that the reflected difference is handled through Pucci inequalities; no additive identity for \(\Mplus\) is used. Compactness of normalized tangential translates prevents loss of contact at infinity.

Our second main result concerns uniform a priori bounds for positive
solutions of \cref{eq:problem1}. Under the geometric assumptions on the positive,
negative, and zero regions of $a$, together with the prescribed power-type
behavior of the weight near the relevant interfaces, we prove that, for
every bounded interval $I\subset\mathbb{R}$, there exists a constant $C>0$ such that $ \|u\|_{L^\infty(\Omega)}\leq C$ for every positive solution of \cref{eq:problem1} with $\mu\in I$, provided that $p$ lies in the exponent range specified in the main a priori estimate.
An important preliminary observation is that the spectral restriction on
$\mu$ is not an additional assumption on $I$. Indeed, the principal
half-eigenvalue argument shows that every positive solution necessarily satisfies $\mu<\sigma_*$ where $\sigma_*$ is determined by the principal positive half-eigenvalues of the relevant components of the zero and nonnegative regions of the weight.

The proof is by contradiction and follows a blow-up analysis adapted to the geometry of the zero set of $a$. Suppose that $\{(\mu_n,u_n)\}$ is a sequence of positive solutions such that $\{\mu_n\}\subset I$ and $\|u_n\|_{L^\infty(\Omega)}
    \longrightarrow +\infty.$
For each $n$, choose $x_n\in\Omega$ such that $\frac{u_n}{\psi}(x_n)=s_n$ where $\psi$ is defined in Lemma \ref{lem:positive-barrier}. The local maximization of $\frac{u_n}{\psi}$ allows the maximum points to be chosen outside the interior of the zero region of $a$. Then we take $u_n(x_n)=M_n,$ after passing to a subsequence, $x_n$ converges to some point of $\overline{\Omega}$, and the remainder of the argument consists in excluding every possible location of this limiting point.

When $x_n$ remains in a compact subset of $\Omega^+$, the usual normalization with scale $\rho_n=M_n^{-\frac{p-1}{2}}$ produces, after passing to a subsequence, a bounded nontrivial solution of a constant-coefficient Pucci equation in $\mathbb{R}^N$. Such a solution is excluded by the corresponding whole-space Liouville theorem. The situation is different when $x_n$ approaches an interface $\Gamma$ at which the positive part of the weight vanishes according to a power law, namely, $ a(x)\sim \alpha(x)d(x,\Gamma)^\gamma.$ In this case the natural blow-up scale is $ \rho_n=M_n^{-\frac{p-1}{2+\gamma}}.$
If the distance from $x_n$ to the interface is comparable with
$\rho_n$, then the rescaled coefficient retains the vanishing profile of $a$, and the limiting problem takes the weighted form
\[
    -\mathcal{M}^{+}_{\lambda,\Lambda}(D^2v)
    =
    h(y_1)v^p
    \qquad \text{in } \mathbb{R}^N.
\]
This is precisely the regime in which the weighted Liouville theorem
proved in \cref{sec:weighted-liouville} enters the a priori argument. On the other hand, if
the distance to the interface is large compared with the blow-up scale,
a secondary rescaling converts the problem into a constant-coefficient
entire equation, which is ruled out by the corresponding unweighted
Liouville theorem given in \cite{CutriLeoni2000}. The same distinction is used for a thin
positive--negative interface, where the limiting weight retains the
signed power profile on the two sides of the interface.
 Concentration near the
interfaces adjacent to the negative region is ruled out using the
corresponding rescaled inequalities and nonexistence arguments.

These alternatives exhaust all possible accumulation points of a
sequence of maximum points. Consequently, an unbounded sequence of
positive solutions cannot occur, and the desired uniform
$L^\infty$ estimate follows. The role of the weighted Liouville theorem
is therefore specific and transparent: it excludes the nonautonomous
entire profiles generated when the blow-up scale is comparable with the
distance to a vanishing or sign-changing interface. Away from such
interfaces, the limiting problems are either constant-coefficient
whole-space or half-space equations, or they are controlled by
comparison in the region where the weight is negative.

As an application of the a priori estimate, the global bifurcation
argument developed in Section~5 yields the existence of positive
solutions. More precisely, the positive continuum bifurcating from $\bigl(\mu_1^+(\Omega),0\bigr)$
cannot remain bounded because of the global bifurcation alternative, while the spectral upper bound and the uniform a priori estimate force its parameter projection to extend indefinitely to the left. Consequently, under the hypotheses guaranteeing these two properties,
the problem admits at least one positive solution for every
$\mu<\mu_1^+(\Omega).$\\

The paper is organized as follows. Section~2 introduces the geometric
and structural hypotheses and states the main results. Section~3
collects the viscosity stability, regularity, comparison, principal
half-eigenvalue, and unweighted Liouville results used later.
Section~4 establishes the weighted Liouville theorem by a moving-plane
argument adapted to Pucci's maximal operator. Section~5 develops the
spectral preliminaries and the localization of maximum points, carries
out the blow-up analysis in the interior, near the zero and
sign-changing interfaces, and at the physical boundary, and then
derives the uniform a priori estimate and the existence result by
global bifurcation.

\section{Setting, structural hypotheses, and main results}
\label{sec:setting}

\subsection{The Dirichlet problem and the solution class}

Let \(N\ge2\), let \(0<\lambda\le\Lambda\), and let
\(\Om\subset\R^N\) be a bounded domain.  We study the positive Dirichlet
problem
\begin{equation}\label{eq:problem}
\begin{cases}
 -\mathcal{M}_{\lambda,\Lambda}^{+}(D^{2}u)=\mu u+a(x)u^{p} & \text{in }\Om,\\
 u>0 & \text{in }\Om,\\
 u=0 & \text{on }\partial\Om,
\end{cases}
\end{equation}
where \(p>1\) and \(\mu\in\R\).  The symbol \(\mu\) is reserved for the
parameter in \cref{eq:problem}; \(\lambda\) and \(\Lambda\) always denote
the ellipticity constants.  We use the sign regions
\begin{equation}\label{eq:sign-regions}
 \OmP:=\{x\in\Om:a(x)>0\},\qquad
 \OmM:=\{x\in\Om:a(x)<0\},\qquad
 \OmZ:=\{x\in\Om:a(x)=0\}.
\end{equation}

Unless explicitly stated otherwise, a solution of \cref{eq:problem} means
a function \(u\in C(\overline\Om)\) satisfying the equation in the
viscosity sense and the boundary condition pointwise.

\begin{assumption}[Domain and global weight]\label{ass:domain}
The domain \(\Om\) is bounded and connected, and \(\partial\Om\) is of
class \(C^2\).  The weight \(a\in C(\overline\Om)\) changes sign; in
particular, both \(\OmP\) and \(\OmM\) are nonempty.  Moreover,
\begin{equation}\label{eq:zero-set-interior}
\left\{
\begin{aligned}
    \OmZ&\Subset\Om\\
 a>0~~&\text{or}~~a<0~~\text{on}~~\partial\Omega.
\end{aligned}
\right.
 \end{equation}
\end{assumption}

\subsection{Geometry of the neutral set}\label{gem_a}

We now state explicitly the geometry of $\Omega^0$ which is defined as
\begin{equation}\label{eq:zero-components}
 \OmZ=Z_{+}\cup\,Z_{-}\cup\,Z_{-}^{+},
\end{equation}
where $Z_{+}$, $Z_{-}$ and $Z_{-}^{+}$ are connected components of $\Omega^{0}$.  The subscripts record which sign region is adjacent to the component; they do not mean that the weight has that sign on the component itself.

\begin{assumption}[Neutral components]\label{ass:geometry}
The interiors \(Z_+^{\circ}\) and \(Z_-^{\circ}\) are nonempty connected
domains and satisfy $\partial Z_+\subset\overline{\OmP}$ and $\partial Z_-\subset\overline{\OmM}$. \\
The separating component \(Z_{-}^{+}\) has one of the following forms.
\begin{enumerate}[label=\textnormal{(G\arabic*)},leftmargin=2.6em]
\item \emph{Thin-interface case:}  If $Z_{-}^{+}$ has empty interior,
  we assume that
\[
Z_{-}^{+}=\Gamma=\overline{\Omega^{+}}\cap\overline{\Omega^{-}},
\]
where $\Gamma$ is a smooth manifold of dimension $N-1$.
 \item \emph{Thick-layer case:  } If $Z_{-}^{+}$ has nonempty interior,
  we assumed that
   \[
      \partial Z_{-}^{+}=\Gamma_+\cup\,\Gamma_-,
   \]
where \(\Gamma_+\subset\overline{\Omega^{+}}\) and \(\Gamma_-\subset\overline{\Omega^{-}}\) are disjoint smooth manifolds of dimension $(N-1).$
   
\end{enumerate}

Let us denote
\begin{equation}\label{eq:filled-positive-domain}
 G_+:=\operatorname{int}
      ( \overline{\OmP\cup Z_+\cup Z_{-}^{+}})
\end{equation}
is assumed to be connected and to have a \(C^2\) boundary.
\end{assumption}
\begin{figure}[ht]
\centering

\begin{tikzpicture}[scale=.75]

\definecolor{myred}{RGB}{220,120,120}
\definecolor{myblue}{RGB}{125,160,210}
\definecolor{myyellow}{RGB}{235,210,115}
\definecolor{mygreen}{RGB}{150,200,155}
\definecolor{mypurple}{RGB}{180,145,210}

\fill[myred]    (0,0) ellipse [x radius=6.0,  y radius=4.5];
\fill[myblue]   (0,0) ellipse [x radius=5.25, y radius=3.90];
\fill[myred]    (0,0) ellipse [x radius=4.40, y radius=3.20];
\fill[myyellow] (0,0) ellipse [x radius=3.45, y radius=2.45];
\fill[mygreen]  (0,0) ellipse [x radius=2.55, y radius=1.80];
\fill[mypurple] (0,0) ellipse [x radius=1.65, y radius=1.15];
\fill[mygreen]  (0,0) ellipse [x radius=0.85, y radius=0.60];

\draw[thick] (0,0) ellipse [x radius=6.0,  y radius=4.5];
\draw[thick] (0,0) ellipse [x radius=5.25, y radius=3.90];
\draw[thick] (0,0) ellipse [x radius=4.40, y radius=3.20];
\draw[thick] (0,0) ellipse [x radius=3.45, y radius=2.45];
\draw[thick] (0,0) ellipse [x radius=2.55, y radius=1.80];
\draw[thick] (0,0) ellipse [x radius=1.65, y radius=1.15];
\draw[thick] (0,0) ellipse [x radius=0.85, y radius=0.60];

\def\ang{55}

\node at ({5.625*cos(\ang)},{4.20*sin(\ang)}) {\small $\Omega^{-}$};
\node at ({4.825*cos(\ang)},{3.55*sin(\ang)}) {\small  $Z_{-}$};
\node at ({3.925*cos(\ang)},{2.825*sin(\ang)}) {\small  $\Omega^{-}$};
\node at ({3.000*cos(\ang)},{2.125*sin(\ang)}) {\small  $Z_{-}^{+}$};
\node at ({2.100*cos(\ang)},{1.475*sin(\ang)}) {\small  $\Omega^{+}$};
\node at ({1.250*cos(\ang)},{0.875*sin(\ang)}) {\small  $Z_{+}$};
\node at (0,0) {\small $\Omega^{+}$};

\end{tikzpicture}

\caption{Schematic representation of the regions
$\Omega^{+}$, $\Omega^{-}$, $Z_{+}$, $Z_{-}$ and $Z_{-}^{+}$.}
\label{fig:concentric-regions}
\end{figure}
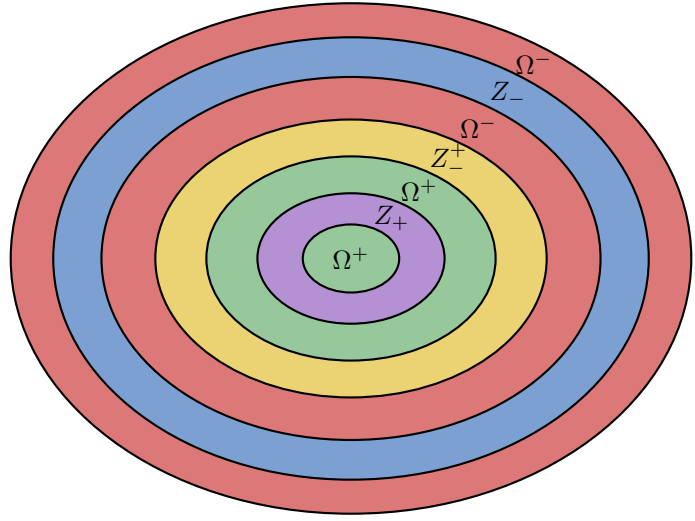

\subsection{Power profiles at the interfaces}
For a closed set $\Sigma\subset\overline{\Omega}$, we denote
\[
d_{\Sigma}(x):=\operatorname{dist}(x,\Sigma).
\]
\begin{assumption}\label{ass:profile}[Power profiles at the interfaces]
In the thick-layer case \textup{(G2)}, the following conditions hold.

\begin{enumerate}
    \item[(i)] Near the interface $\Gamma_{+}$, there exist
    $\gamma>0$ and a positive continuous function $\alpha_{1}$,
    defined in a neighborhood of $\Gamma_{+}$, such that
    \[
    a(x)
    =
    \alpha_{1}(x)\,d_{\Gamma_{+}}(x)^{\gamma},
    \qquad
    x\in\Omega^{+}
    \quad\text{near }\Gamma_{+}.
    \]

    \item[(ii)] Near $\partial Z_{+}$, there exist
    $\gamma_{3}>0$ and a positive continuous function $\alpha_{3}$,
    defined in a neighborhood of $\partial Z_{+}$, such that
    \[
    a(x)
    =
    \alpha_{3}(x)\,d_{\partial Z_{+}}(x)^{\gamma_{3}},
    \qquad
    x\in\Omega^{+}
    \quad\text{near }\partial Z_{+}.
    \]

    \item[(iii)] Near $\Gamma_{-}$, there exist constants
    $0<m<M$ and an exponent $\gamma_{4}>0$ such that
    \[
    m
    \leq
    \frac{-a(x)}
    {d_{\Gamma_{-}}(x)^{\gamma_{4}}}
    \leq
    M,
    \qquad
    x\in\Omega^{-}
    \quad\text{near }\Gamma_{-}.
    \]

    \item[(iv)] Near $\partial Z_{-}$, there exist constants
    $0<m<M$ and an exponent $\gamma_{4}>0$ such that
    \[
    m
    \leq
    \frac{-a(x)}
    {d_{\partial Z_{-}}(x)^{\gamma_{4}}}
    \leq
    M,
    \qquad
    x\in\Omega^{-}
    \quad\text{near }\partial Z_{-}.
    \]
\end{enumerate}

In the thin-interface case \textup{(G1)}, there exist positive continuous
functions $\alpha_{1}$ and $\alpha_{2}$, defined in a neighborhood of
$\Gamma$, and exponent $\gamma_{2}>0$ such that
\begin{equation}\label{eq:thin}
    \left\{
    \begin{aligned}
        a(x)&=\alpha_{1}(x)\,d_{\Gamma}(x)^{\gamma_{2}},\qquad x\in\Omega^{+}\quad\text{near }\Gamma,\\
        a(x)&=-\alpha_{2}(x)\,d_{\Gamma}(x)^{\gamma_{2}},
\qquad x\in\Omega^{-}\quad\text{near }\Gamma.
    \end{aligned}
    \right.
\end{equation}

\end{assumption}

\subsection{The weighted entire problem}

The blow-up scale at a vanishing interface leads to
\begin{equation}\label{eq:weighted-entire}
\left\{
\begin{aligned}
     -\Mplus(D^{2}v)&=h(x_1)f(v)
 \quad\text{in }\R^N\\
 v&\ge0\\
 v&\in L^{\infty}(\R^N).
\end{aligned}
\right.
\end{equation}

\begin{assumption}[Data for the weighted problem]
\label{ass:liouville-data}
For some \(\alpha\in(0,1)\), the functions \(h\) and \(f\) satisfy:
\begin{enumerate}[label=\textnormal{(L\arabic*)},leftmargin=2.6em]
 \item \(h\in C_{loc}^{\alpha}(\R)
   \cap C^1(\R\setminus\{0\})\), \(h(0)=0\), and
   \[
      h'(t)\ge0\quad~~\text{for}~~t\ne0,\qquad h'(t)>0\quad~~\text{for}~~t>0.
   \]
 \item The behavior at the origin satisfies
   \[
      \limsup_{t\to0^+}\frac{h(t)}{h'(t)}<\infty.
   \]
 \item \(f\in C^1([0,\infty))\),
   \(f(0)=f'(0)=0\), and
   \[
      f(s)>0,\qquad f'(s)\ge0\qquad~~\text{for}~~s>0.
   \]
 \item The behavior at the origin satisfies
   \[
      \limsup_{s\to0^+}\frac{s f'(s)}{f(s)}<\infty.
   \]
\end{enumerate}
\end{assumption}

\begin{theorem}[Weighted Liouville theorem]
\label{thm:weighted-liouville}
Let the assumption \ref{ass:liouville-data} hold. Then  Every bounded nonnegative viscosity
solution of \cref{eq:weighted-entire} is either identically zero or 
strictly increasing in the \(x_1\); more precisely, in the
nontrivial case,
\[
 v(x_1+\tau,x')>v(x_1,x')
 \qquad\text{for every }\tau>0\text{ and }(x_1,x')\in\R^N.
\]
If, in addition,
\begin{equation}\label{eq:h-growth}
 h(t)\longrightarrow+\infty\qquad\text{as }t\to+\infty,
\end{equation}
then \(v\equiv0\).
\end{theorem}

\begin{corollary}[Power-interface models]\label{cor:power-weight}
Let \(p>1\), \(\gamma>0\), \(c_+>0\), and \(c_-\ge0\).  The only bounded
nonnegative viscosity solution of
\[
 -\Mplus(D^2v)
   =\bigl(c_+\pos{x_1}^{\gamma}
          -c_-\negpart{x_1}^{\gamma}\bigr)v^p
 \quad\text{in }\R^N
\]
is \(v\equiv0\).
\end{corollary}

\subsection{Spectral threshold and exponent range}
If \(D\subset\R^N\) is a bounded \(C^2\) domain, denote by
\(\mu_1^{+}(D)\) the positive principal half-eigenvalue of
\(-\Mplus\), characterized by the existence of
\(\phi_D>0\) in \(D\) satisfying
\begin{equation}\label{eq:principal-half-eigenvalue}
\begin{cases}
 -\Mplus(D^2\phi_D)=\mu_1^{+}(D)\phi_D & \text{in }D,\\
 \phi_D=0 & \text{on }\partial D.
\end{cases}
\end{equation}
Define 
\begin{equation}\label{eq:spectral-threshold}
 \sigma_*:=\min\bigl\{\mu_1^{+}(G^+),
                 \mu_1^{+}(Z_-)\bigr\}.
\end{equation}

\begin{proposition}[Necessary spectral restriction]
\label{prop:necessary-spectral-restriction}
If \cref{eq:problem} has a positive solution, then
\begin{equation}\label{eq:necessary-spectral-restriction}
 \mu<\sigma_*.
\end{equation}
\end{proposition}

\subsection{Uniform a priori bound}
\begin{theorem}\label{thm:apriori}[Uniform \(L^\infty\) estimate]
Suppose that assumptions
\ref{ass:domain},\ref{ass:geometry},\ref{ass:profile}
hold and
\begin{equation}\label{eq:p-range}
 1<p\leq\psafe,
\end{equation}
$\mu\in \Ipar,$ where $\Ipar$ is a bounded interval in $\mathbb{R}$ such that $\sup\Ipar<\sigma_*$
Then there exists a constant \(C>0\) such that any positive solution of \eqref{eq:problem1} with $\mu\in\Ipar$ satisfies 
\begin{equation}\label{eq:apriori-bound}
 \norm{u}_{L^\infty(\Om)}\le C.
\end{equation}

\end{theorem}

\section{Preliminary results}\label{sec:preliminaries}

This section fixes the analytical conventions and records precisely the
results used in the moving-plane and blow-up arguments.  All differential
inequalities are understood in the viscosity sense unless a stronger
solution class is explicitly stated.  We prove the algebraic and scaling
facts that are special to the present problem and cite the standard
regularity, maximum-principle, and Liouville inputs.  This separation is
important: the Pucci operator is nonlinear, so identities based on linear
superposition must not be used.

\subsection{Pucci operators and their algebra}

Let
\[
 \mathcal A_{\lambda,\Lambda}
 :=\{A\in\Sn:\lambda I\le A\le\Lambda I\}.
\]
For \(X\in\Sn\), the extremal operators are
\begin{equation}\label{eq:pucci-sup-inf}
 \Mplus(X)=\sup_{A\in\mathcal A_{\lambda,\Lambda}}\operatorname{tr}(AX),
 \qquad
 \Mminus(X)=\inf_{A\in\mathcal A_{\lambda,\Lambda}}\operatorname{tr}(AX).
\end{equation}
Equivalently, if \(e_1,\ldots,e_N\) are the eigenvalues of \(X\), then
\begin{align}
 \Mplus(X)
   &=\Lambda\sum_{e_i>0}e_i+\lambda\sum_{e_i<0}e_i,
 \label{eq:pucci-plus-eigenvalues}\\
 \Mminus(X)
   &=\lambda\sum_{e_i>0}e_i+\Lambda\sum_{e_i<0}e_i.
 \label{eq:pucci-minus-eigenvalues}
\end{align}

\begin{proposition}[Lemma 2.10\cite{CaffarelliCabre1995}]\label{prop:pucci-algebra}
For every \(X,Y\in\Sn\), every \(s\ge0\), and every orthogonal matrix
\(Q\),
\begin{align}
 \Mplus(-X)&=-\Mminus(X),
 \label{eq:pucci-duality}\\
 \Mplus(sX)&=s\Mplus(X),\qquad
 \Mminus(sX)=s\Mminus(X),
 \label{eq:pucci-homogeneity}\\
 \Mplus(Q^TXQ)&=\Mplus(X),\qquad
 \Mminus(Q^TXQ)=\Mminus(X),
 \label{eq:pucci-rotation}\\
 \Mminus(X-Y)&\le \Mplus(X)-\Mplus(Y)
                \le\Mplus(X-Y),
 \label{eq:pucci-difference-plus}\\
 \Mminus(X-Y)&\le \Mminus(X)-\Mminus(Y)
                \le\Mplus(X-Y).
 \label{eq:pucci-difference-minus}
\end{align}
\end{proposition}

\begin{lemma}[Lemma 3.2\cite{armstrong2009principal}]\label{lem:viscosity-difference}
Let \(\tilde D\subset\R^N\) be open, let \(f,g\in C(\tilde D)\), and let
\(u,v\in C(\tilde D)\) solve
\[
 -\Mplus(D^2u)=f(x),\qquad -\Mplus(D^2v)=g(x)
 \quad\text{in }\tilde D.
\]
Then \(w=u-v\) satisfies
\begin{equation}\label{eq:viscosity-difference}
 -\Mplus(D^2w)\le f-g\le-\Mminus(D^2w)
 \quad\text{in }\tilde D.
\end{equation}
In particular, \(\Mplus(D^2u)-\Mplus(D^2v)\) must never be replaced by
\(\Mplus(D^2(u-v))\) as an identity.
\end{lemma}

\subsection{Viscosity convention and stability}

\begin{definition}[Sign convention]\label{def:viscosity}
Let \(G:D\times\R\to\R\) be continuous and put
\[
 F(x,r,X):=-\Mplus(X)-G(x,r).
\]
An upper semicontinuous function \(u\) is a viscosity subsolution of
\(F=0\) if, whenever \(\phi\in C^2(D)\) touches \(u\) from above at
\(x_0\in D\),
\[
 -\Mplus(D^2\phi(x_0))-G(x_0,u(x_0))\le0.
\]
A lower semicontinuous function is a viscosity supersolution if the reverse
inequality holds for every test function touching from below.  A continuous
function $r$ is called viscosity solution if it is both sub and supersolution.
\end{definition}

For \cref{eq:problem}, one takes
\(G(x,r)=\mu r+a(x)r^p\) for \(r\ge0\).  As a sign check, at a smooth local
maximum of \(u\), one has \(D^2u\le0\), hence
\(\Mplus(D^2u)=\lambda\operatorname{tr}(D^2u)\le0\) and
\(-\Mplus(D^2u)\ge0\).

\subsection{H\texorpdfstring{\"o}{o}lder estimates, Harnack inequalities,
and compactness}

\begin{theorem}[Proposition~4.10\cite{CaffarelliCabre1995}~Interior H\texorpdfstring{\"o}{o}lder estimate]
\label{thm:interior-holder}
There exist \(\alpha\in(0,1)\) and \(C\ge1\), depending only on
\(N,\lambda,\Lambda\), with the following property.  If
\(u\in C(B_{2R}(x_0))\) satisfies
\[
 \Mminus(D^2u)\le K,
 \qquad
 \Mplus(D^2u)\ge-K
 \quad\text{in }B_{2R}(x_0),
\]
then
\begin{equation}\label{eq:interior-holder}
 [u]_{C^\alpha(B_R(x_0))}
 \le C R^{-\alpha}
 \bigl(\norm{u}_{L^\infty(B_{2R}(x_0))}+R^2K\bigr).
\end{equation}
\end{theorem}

\begin{theorem}[Theorem~4.3\cite{CaffarelliCabre1995}~Harnack inequality]\label{lem:harnack}
Let \(u\ge0\) solve
\[
 -\Mplus(D^2u)=q(x)u\quad\text{in }B_{2R}(x_0),
 \qquad q\in C(B_{2R}(x_0))\cap L^\infty(B_{2R}(x_0)).
\]
Then
\begin{equation}\label{eq:harnack}
 \sup_{B_R(x_0)}u\le C\inf_{B_R(x_0)}u,
\end{equation}
where \(C\) depends only on
\(N,\lambda,\Lambda\), and \(R^2\norm q_{L^\infty}\).
\end{theorem}

\begin{corollary}[Bounded homogeneous entire solutions]
\label{cor:bounded-homogeneous}
If \(u\in C(\R^N)\cap L^\infty(\R^N)\) satisfies
\(-\Mplus(D^2u)=0\) in \(\R^N\), then \(u\) is constant.
\end{corollary}

\begin{proof}
Fix \(x,y\in\R^N\).  Apply \cref{eq:interior-holder} with \(K=0\) in a
ball of radius \(R\) containing \(x\) and \(y\).  It gives
\[
 |u(x)-u(y)|\le C\norm u_{L^\infty(\R^N)}
       \left(\frac{|x-y|}{R}\right)^\alpha.
\]
Letting \(R\to\infty\) proves the claim.  Notice that the strong maximum
principle alone would not prove this conclusion.
\end{proof}

\begin{proposition}[Theorem~3.8\cite{caffarelli1996viscosity}, Theorem~2.4\cite{QuaasSirakov2006}]
\label{lem:compactness}
Let \(D_n\subset\R^N\) be domains that converge locally to a domain \(D\),
in the sense that every compact subset of \(D\) is contained in \(D_n\) for
all sufficiently large \(n\).  Suppose
\[
 \norm{v_n}_{L^\infty(D_n)}\le C_0,
 \qquad
 -\Mplus(D^2v_n)=g_n\quad\text{in }D_n,
\]
and, for every compact \(K\Subset D\),
\(\sup_n\norm{g_n}_{L^\infty(K)}<\infty\).  Then a subsequence converges
locally uniformly in \(D\) to a continuous function \(v\).  If
\(g_n\to g\) locally uniformly, then
\[
 -\Mplus(D^2v)=g\quad\text{in }D.
\]

If, in addition, \(D_n\) converge in uniformly \(C^2\) local boundary
charts to a half-space, \(v_n=0\) on the corresponding boundary pieces,
and the right-hand sides are uniformly bounded there, then the convergence
holds locally up to the limiting boundary and \(v=0\) on that boundary.
\end{proposition}

\subsection{Comparison and maximum principles}

\begin{theorem}[Proper comparison]\label{lem:comparison}
Let \(D\) be bounded and let \(G\in C(\overline D\times\R)\) be
nondecreasing in its second variable.  Suppose \(u\in\operatorname{USC}(\overline D)\)
is a subsolution and \(v\in\operatorname{LSC}(\overline D)\) is a
supersolution of
\[
 -\Mplus(D^2z)+G(x,z)=0\quad\text{in }D.
\]
If \(u\le v\) on \(\partial D\), then \(u\le v\) in \(D\).
\end{theorem}

This is the standard viscosity comparison theorem
\cite{CrandallIshiiLions1992,CaffarelliCabre1995}.  The equation in
\cref{eq:problem} corresponds instead to
\[
 G(x,r)=-\mu r-a(x)r^p,
\]
which need not be nondecreasing in \(r\).  Consequently,
\cref{lem:comparison} cannot be invoked directly for the indefinite problem.\\

\begin{theorem}[Strong maximum principle and Hopf lemma]
\label{lem:smp-hopf}
Let \(D\) be regular domain, let \(c_0\in L^\infty(D)\), and let \(w\in C(D)\) be
nonnegative and satisfy
\begin{equation}\label{eq:smp-inequality}
 \Mminus(D^2w)+c_0w\le0\quad\text{in }D.
\end{equation}
Then either \(w\equiv0\) or \(w>0\) throughout \(D\).  Moreover, if the second
alternative holds, \(x_0\in\partial D\) is a \(C^2\) boundary point,
\(D\) has an interior tangent ball at \(x_0\), and
\(w\in C(D\cup\{x_0\})\) with \(w(x_0)=0\), then
\begin{equation}\label{eq:hopf-form}
 \limsup_{t\to0^+}\frac{w(x_0)-w(x_0-t\nu)}{t}<0,
\end{equation}
where \(\nu\) is the outer unit normal.  
\end{theorem}

For a general strong maximum principle for degenerate convex elliptic operators; see
\cite{MR1710100}.

\subsection{Principal half-eigenvalues}
\begin{proposition}[Proposition 1.1\cite{MR2124162}]\label{prop:first_halfeigenvalue}
Let $G\subset\mathbb{R}^{N}$ be a bounded regular domain. Then there exist two positive constants $\mu_{1}^{+}(G),$ $\mu_{1}^{-}(G),$
called the \emph{first half-eigenvalues} of the Pucci operator $\mathcal{M}^{+}_{\lambda,\Lambda}$, satisfying the following properties.
\begin{enumerate}
\item[(i)] There exist functions $\phi_{1}^{+},\,\phi_{1}^{-}
\in C^{2}(G)\cap C(\overline{G}),$ such that $(\mu_{1}^{+},\phi_{1}^{+})$ and $(\mu_{1}^{-},\phi_{1}^{-})$ solve
\begin{equation}\label{eq:halfeigenvalue}
\left\{
\begin{aligned}
-\mathcal{M}^{+}_{\lambda,\Lambda}(D^{2}\phi)&= \mu\,\phi &&\text{in}~G,\\
\phi&=0 &&\text{on }\partial G,
\end{aligned}
\right.
\end{equation}
and $\phi_{1}^{+}>0,$ $\phi_{1}^{-}<0$ in $G.$ Moreover, these half-eigenvalues are simple. More precisely, every positive solution of
\eqref{eq:halfeigenvalue} is of the form $(\mu_{1}^{+},\alpha\phi_{1}^{+}),$ with $\alpha>0,$ and an analogous statement holds for negative solutions.
\item[(ii)] The first half-eigenvalues admit the variational characterization
\[\mu_{1}^{+}=\inf_{A\in\mathcal{A}}\mu_{1}(A),\qquad\mu_{1}^{-}=\sup_{A\in\mathcal{A}}
\mu_{1}(A),\]
where $\mathcal{A}$ denotes the class of all symmetric measurable matrix-valued functions satisfying $\lambda I\leq A(x)\le\Lambda I$ a.e. in $G,$ and $\mu_{1}(A)$ is the principal eigenvalue of the corresponding linear uniformly elliptic operator in non-divergence form associated with $A.$
\item[(iii)] The first half-eigenvalues can also be characterized by
\begin{equation}\label{1eq:halfeigenvalue}
\left\{
\begin{aligned}
&\mu_{1}^{+}(G)=\sup_{u>0}
\operatorname*{ess\,inf}_{G}
\left(
\frac{-\mathcal{M}^{+}_{\lambda,\Lambda}(D^{2}u)}{u}
\right),\\
&\text{and}\\
&\mu_{1}^{-}(G)
=\sup_{u<0}\operatorname*{ess\,inf}_{G}\left(
\frac{-\mathcal{M}^{+}_{\lambda,\Lambda}(D^{2}u)}{u}
\right),
\end{aligned}
\right.
\end{equation}
where the supremum is taken over all functions $u\in W^{2,N}_{\mathrm{loc}}(G)\cap C(\overline{G}).$
\item[(iv)] Finally, the first half-eigenvalues admit the equivalent characterization
\[\mu_{1}^{+}=\sup\left\{\mu\in\mathbb{R}~:~\exists\,\phi>0
~\text{in}~~G~\text{satisfying}~
\mathcal{M}^{+}_{\lambda,\Lambda}(D^{2}\phi)+\mu\phi\le0\right\},\]
and
\[\mu_{1}^{-}=\sup\left\{\mu\in\mathbb{R}
~:~\exists\,\phi<0
~\text{in}~G~\text{satisfying}~
\mathcal{M}^{+}_{\lambda,\Lambda}(D^{2}\phi)+\mu\phi\ge0
\right\}.\]
\end{enumerate}
\end{proposition}




\subsection{Exact rescaling and limiting geometry}

The following statements are the external nonexistence inputs used in a
standard blow-up argument.  We record their exponent ranges exactly.

\begin{theorem}[Whole-space Liouville input; Cutri--Leoni]
\label{thm:whole-space}
Assume \(N\ge3\) and define
\begin{equation}\label{eq:cutri-leoni-exponent}
 \NCL:=1+\frac{\lambda}{\Lambda}(N-1),
 \qquad
 \pCL:=\frac{\NCL}{\NCL-2}.
\end{equation}
Let \(c>0\) and \(1<p\le\pCL\)  (or $1<p<\infty $ if $\NCL\leq2${)}.  If \(v\ge0\) is a continuous viscosity
supersolution of
\begin{equation}\label{eq:whole-space-liouville-inequality}
 -\Mplus(D^2v)\ge c\,v^p\quad\text{in }\R^N,
\end{equation}
then \(v\equiv0\).
\end{theorem}

This is the Pucci maximal-operator case of
\cite[Theorem~4.1]{CutriLeoni2000}; it is also stated in precisely this form
in \cite[Theorem~1.3]{QuaasSirakov2006}.  The coefficient \(c\) is
inessential: replacing \(v\) by \(c^{-\frac{1}{p-1}}v\) reduces it to \(c=1\).

\begin{theorem}[Theorem 1.5\cite{QuaasSirakov2006}]
\label{thm:half-space}
Assume \(N\ge3\) and set
\begin{equation}\label{eq:half-space-exponent}
 \pHS:=
 \begin{cases}
  \dfrac{\lambda(N-2)+\Lambda}{\lambda(N-2)-\Lambda},
     &\lambda(N-2)>\Lambda,\\[7pt]
  +\infty,&\lambda(N-2)\leq\Lambda.
 \end{cases}
\end{equation}
Let \(c>0\) and \(1<p\le\pHS\).  The problem
\begin{equation}\label{eq:half-space-liouville-problem}
 \begin{cases}
  -\Mplus(D^2v)=c\,v^p&\text{in }\R^N_+:=\{x_N>0\},\\
  v\ge0&\text{in }\R^N_+,\\
  v=0&\text{on }\partial\R^N_+
 \end{cases}
\end{equation}
has no nonzero bounded solution.
\end{theorem}

  It is
written there for classical bounded solutions.  A bounded continuous
viscosity solution of \cref{eq:half-space-liouville-problem} is locally
H\"older continuous; hence \(v^p\) is locally H\"older continuous, and the
interior \(C^{2,\alpha}\) regularity for the convex Pucci operator places it
in the source's classical class.  The same reference records both this
interior regularity and H\"older continuity up to a zero boundary.





\textbf{Motivation for Nonlinear Liouville-type Theorem :}\\
Since we are working on the problem 
\begin{equation*}
    \left\{
    \begin{aligned}
        -\mathcal{M}_{\lambda,\Lambda}^{+}(D^2u)&=\mu u+a(x)u^p&&\text{in}&&\Omega,\\
        u&=0&&\text{on}&&\partial\Omega.
    \end{aligned}
    \right.
\end{equation*}
Suppose, for
example, that a sequence of maximum points $x_n$ approaches a smooth
interface $\Gamma$ on which $a$ vanishes and that, on one side of
$\Gamma$,
\[
a(x)=\alpha(x)[d(x,\Gamma)]^{\gamma},
\qquad \gamma>0,
\]
where $\alpha$ is positive and continuous. Setting
\[
M_n=u_n(x_n),\qquad
\rho_n=M_n^{-\frac{p-1}{2+\gamma}},
\]
and introducing the rescaled functions
\[
v_n(y)=\frac{u_n(x_n+\rho_n y)}{M_n},
\]
one is naturally led, after passing to a subsequence, to an entire
limiting problem of the form
\[
-\mathcal{M}^{+}_{\lambda,\Lambda}(D^2v)
       =h(y_1)v^p
       \qquad\text{in }\mathbb{R}^{N}.
\]
Thus, the usual autonomous Liouville theorems are no longer
sufficient near the interface where $a$ vanishes. The weighted
Liouville theorem proved in this paper is precisely the tool needed
to exclude such limiting profiles.
\section{The weighted Liouville theorem}\label{sec:weighted-liouville}

In this section we prove \cref{thm:weighted-liouville}.  The moving-plane
strategy is inspired by the semilinear argument of Du and Li
\cite{DuLi2005}.  Two points require special care for the Pucci operator:
the reflected difference satisfies extremal inequalities rather than a
linear equation, and the auxiliary function used on the unbounded reflected
half-space must be a strict supersolution for \(\Mplus\).
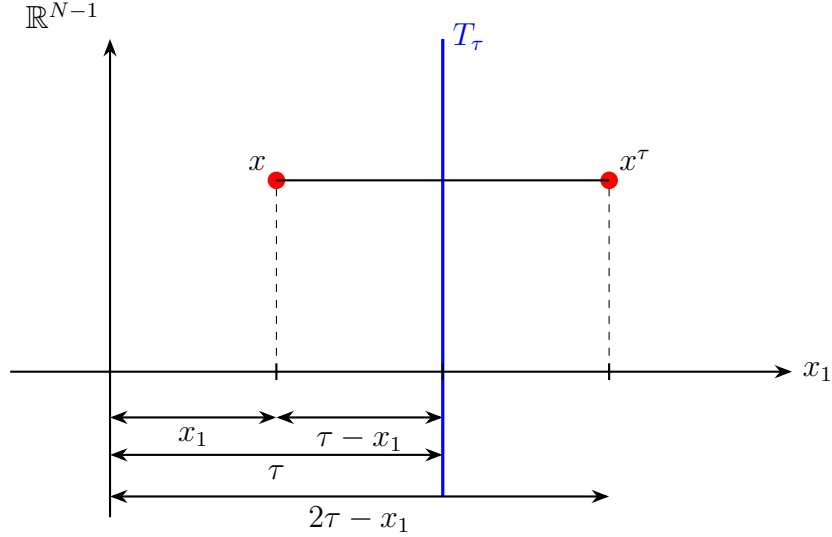
\begin{figure}[ht]
\centering

\begin{tikzpicture}[
    scale=1.1,
    >=Stealth,
    every node/.style={font=\large}
]


\coordinate (O)  at (0,0);
\coordinate (X)  at (2,2.3);
\coordinate (T)  at (4,0);
\coordinate (XR) at (6,2.3);


\draw[->, thick]
    (-1.2,0) -- (8.2,0)
    node[right] {$x_1$};

\draw[->, thick]
    (0,-1.75) -- (0,4.0)
    node[above left] {$\mathbb{R}^{N-1}$};

\draw[blue, very thick]
    (4,-1.5) -- (4,4.0);

\node[blue, above right] at (4,3.75)
    {$T_\tau$};


\fill[red] (X) circle (3pt);
\node[above left] at (X) {$x$};

\fill[red] (XR) circle (3pt);
\node[above right] at (XR) {$x^\tau$};

\draw[thick]
    (X) -- (XR);


\draw[dashed]
    (2,0) -- (2,2.3);

\draw[dashed]
    (6,0) -- (6,2.3);

\draw[thick] (2,-0.10) -- (2,0.10);
\draw[thick] (4,-0.10) -- (4,0.10);
\draw[thick] (6,-0.10) -- (6,0.10);


\draw[<->, thick]
    (0,-0.55) -- (2,-0.55);

\node[below] at (1,-0.55)
    {$x_1$};


\draw[<->, thick]
    (2,-0.55) -- (4,-0.55);

\node[below] at (3,-0.55)
    {$\tau-x_1$};


\draw[<->, thick]
    (0,-1.00) -- (4,-1.00);

\node[below] at (2,-1.00)
    {$\tau$};


\draw[<->, thick]
    (0,-1.50) -- (6,-1.50);

\node[below] at (3,-1.50)
    {$2\tau-x_1$};

\end{tikzpicture}

\caption{Reflection of the point $x$ with respect to the hyperplane $T_\tau$.}
\label{fig:reflection-Tmu}
\end{figure}

\subsection{Regularity and the reflected inequality}

We first justify working pointwise.  Let \(v\) be a bounded nonnegative
viscosity solution of \cref{eq:weighted-entire}.  Its right-hand side is
locally bounded, so the interior estimates give
\(v\in C^{1,\beta}_{\mathrm{loc}}(\R^N)\) for some \(\beta\in(0,1)\).
 From L1 of Assumption \ref{ass:liouville-data} \(h\in C^{0,\alpha}_{\mathrm{loc}}(\R)\) and \(f\in C^1\), the
function \(h(x_1)f(v)\) is locally H\"older continuous. The interior
Schauder estimate for the convex Pucci operator therefore yields
\begin{equation}\label{eq:liouville-regularity}
 v\in C^{2,\beta}_{\mathrm{loc}}(\R^N),
\end{equation}
after possibly decreasing \(\beta\); see
\cite{CaffarelliCabre1995} and
\cite[Theorem~2.3]{QuaasSirakov2006}.

\noindent Define
\[
 c(x):=
 \begin{cases}
  h(x_1)\dfrac{f(v(x))}{v(x)},&v(x)>0,\\[6pt]
  0,&v(x)=0.
 \end{cases}
\]
Because \(f(s)/s\to f'(0)=0\) as \(s\to0^{+}\), we have
\(c\in L^\infty_{\mathrm{loc}}(\R^N)\).  Thus the strong minimum
principle, applied locally, shows that
\begin{equation}\label{eq:zero-or-positive}
 v\equiv0\quad\text{or}\quad v>0\text{ in }\R^N.
\end{equation}

We henceforth assume that \(v>0\).

To apply a nonlinear Liouville-type result using the moving plane method, for \(\tau\in\R\) (see \ref{fig:reflection-Tmu}), set
\begin{equation}\label{eq:reflection-notation}
 \Sigma_\tau:=\{x=(x_1,x')\in\R^N:x_1<\tau\},\qquad
 x^\tau:=(2\tau-x_1,x'),
\end{equation}
and define
\[
 v_\tau(x):=v(x^\tau),\qquad w_\tau(x):=v_\tau(x)-v(x)
 \quad\text{in }\Sigma_\tau.
\]
It is easy to check that
\[
 -\Mplus(D^2v_\tau)=h(2\tau-x_1)f(v_\tau).
\]
Applying Lemma \ref{lem:viscosity-difference}, or equivalently
\cref{eq:pucci-difference-plus} in view of
\cref{eq:liouville-regularity}, we obtain
\begin{equation}\label{eq:reflected-pucci-inequality}
 \mathcal{M}_{\lambda,\Lambda}^{-}(D^2w_\tau)
 \le h(x_1)f(v)-h(2\tau-x_1)f(v_\tau)
 \quad\text{in }\Sigma_\tau.
\end{equation}
This is the basic reflected inequality used throughout the proof.\\

\subsection{A strict Pucci barrier on the reflected half-space}
The following calculation replaces the harmonic logarithmic barrier used
for the Laplacian.
\begin{remark}
Observe that when $\lambda=\Lambda$, Pucci's maximal operator reduces to a
constant multiple of the Laplacian, namely
\[
\mathcal{M}^{+}_{\lambda,\lambda}(D^2u)
=\lambda\,\operatorname{tr}(D^2u)
=\lambda\Delta u.
\]
Hence, in this case the argument reduces to the corresponding semilinear
Laplacian setting. Therefore, in what follows, it is enough to consider the
genuinely fully nonlinear case
\[
0<\lambda<\Lambda.
\]
\end{remark}

\begin{lemma}[Logarithmic Pucci barrier]\label{lem:pucci-log-barrier}
There exists \(\varepsilon_0>0\), depending only on
\(N,\lambda,\Lambda\), such that, for every \(\tau\in\R\),
\begin{equation}\label{eq:barrier-definition}
 g_\tau(x):=\ln(\tau+3-x_1)+\varepsilon_0\sum_{j=2}^N
                 \ln((\tau+3-x_1)^2+x_j^2)
\end{equation}
satisfies
\begin{align}
 g_\tau&>0&&\text{in }\Sigma_\tau,
 \label{eq:barrier-positive}\\
 g_\tau(x)&\longrightarrow+\infty
     &&\text{as }|x|\to\infty,\quad x\in\Sigma_\tau,
 \label{eq:barrier-coercive}\\
 \Mplus(D^2g_\tau)&<0
     &&\text{in }\Sigma_\tau.
 \label{eq:barrier-pucci-sign}
\end{align}
\end{lemma}

\begin{proof}
Since \(x_1<\tau\), we have \(\tau+3-x_1>3\).  Hence
\cref{eq:barrier-positive,eq:barrier-coercive} are immediate. Now, our claim is to see existence of $\varepsilon_0$ for which \cref{eq:barrier-pucci-sign} holds.
It is easy to get the following identities
\begin{equation*}
\left\{
\begin{aligned}
&(i)~\frac{\partial g_{\tau}(x)}{\partial x_1}=-\frac{1}{(\tau+3-x_1)}-\varepsilon_0\sum_{j=2}^N\frac{2 (\tau+3-x_1)}{[(\tau+3-x_1)^2+ x_j^2]}\\
&(ii)~\frac{\partial g_{\tau}(x)}{\partial x_j}=\varepsilon_0\frac{2 x_j}{[(\tau+3-x_1)^2+ x_j^2]}\\
&(iii)~\frac{\partial^2 g_{\tau}(x)}{\partial x_1^2}=-\frac{1}{(\tau+3-x_1)^2}-2\varepsilon_0\sum_{j=2}^N\frac{(\tau+3-x_1)^2- x_j^2}{[(\tau+3-x_1)^2+ x_j^2]^2}\\
 &(iv)~\frac{\partial^2 g_{\tau}(x)}{\partial x_j^2}=2\varepsilon_0\frac{ (\tau+3-x_1)^2- x_j^2}{[(\tau+3-x_1)^2+ x_j^2]^2}\\
&(v)~~\frac{\partial^2 g_{\tau}(x)}{\partial x_1\partial x_j}=4\varepsilon_0\frac{ (\tau+3-x_1)x_j}{[(\tau+3-x_1)^2+ x_j^2]^2}\\
&\text{and for}~i\neq 1~\text{and}~i\neq j~\text{we have}\\
&(vi)~\frac{\partial^2 g_{\tau}(x)}{\partial x_i\partial x_j}=0.
\end{aligned}
\right.
\end{equation*}
By using the above values we find
\[D^2g_{\tau}(x)= \begin{bmatrix}
  -\frac{1}{(\tau+3-x_1)^2}-2\varepsilon_0 \sum_{j=2}^N\frac{(\tau+3-x_1)^2-x_j^2}{[(\tau+3-x_1)^2+ x_j^2]^2} & 4\varepsilon_0\frac{ (\tau+3-x_1)x_2}{[(\tau+3-x_1)^2+ x_2^2]^2} & \dots & 4\varepsilon_0\frac{ (\tau+3-x_1)x_N}{[(\tau+3-x_1)^2+ x_N^2]^2}\\ 
  4\varepsilon_0\frac{ (\tau+3-x_1)x_2}{[(\tau+3-x_1)^2+ x_2^2]^2} & 2\varepsilon_0\frac{ (\tau+3-x_1)^2-x_2^2}{[(\tau+3-x_1)^2+ x_2^2]^2} & \dots & 0\\
  \vdots & \vdots & \vdots & \vdots\\
  4\varepsilon_0\frac{ (\tau+3-x_1)x_N}{[(\tau+3-x_1)^2+ x_N^2]^2} & 0 & \dots & 2\varepsilon_0\frac{ (\tau+3-x_1)^2-x_N^2}{[(\tau+3-x_1)^2+ x_N^2]^2}
\end{bmatrix}.\]
For convenience we can write the Hessian as $D^2g_{\tau}(x)= A+B,$ where 
\begin{equation*}
\begin{aligned}
&A= \begin{bmatrix}
  -\frac{1}{(\tau+3-x_1)^2} & 4\varepsilon_0\frac{ (\tau+3-x_1)x_2}{[(\tau+3-x_1)^2+ x_2^2]^2} & \dots & 4\varepsilon_0\frac{ (\tau+3-x_1)x_N}{[(\tau+3-x_1)^2+ x_N^2]^2}\\ 
  4\varepsilon_0\frac{ (\tau+3-x_1)x_2}{[(\tau+3-x_1)^2+ x_2^2]^2} & 0 & \dots & 0\\
  \vdots & \vdots & \vdots & \vdots\\
  
  4\varepsilon_0\frac{ (\tau+3-x_1)x_N}{[(\tau+3-x_1)^2+ x_N^2]^2} & 0 & \dots & 0
\end{bmatrix}\\
\text{and}\\
&B= \begin{bmatrix}
-2\varepsilon_0 \sum_{j=2}^N\frac{(\tau+3-x_1)^2-x_j^2}{[(\tau+3-x_1)^2+ x_j^2]^2} & 0 & \dots & 0\\ 
0 & 2\varepsilon_0\frac{ (\tau+3-x_1)^2-x_2^2}{[(\tau+3-x_1)^2+ x_2^2]^2} & \dots & 0\\
\vdots & \vdots & \vdots & \vdots\\
0 & 0 & \dots & 2\varepsilon_0\frac{ (\tau+3-x_1)^2-x_N^2}{[(\tau+3-x_1)^2+ x_N^2]^2}
\end{bmatrix}.
\end{aligned}
\end{equation*}
Furthermore we can write $B=C+D,$ where 
\begin{equation*}
\begin{aligned}
&C= \begin{bmatrix}
-2\varepsilon_0 \sum_{j=2}^N\frac{(\tau+3-x_1)^2}{[(\tau+3-x_1)^2+ x_j^2]^2} & 0 & \dots & 0\\ 
0 & 2\varepsilon_0\frac{ (\tau+3-x_1)^2}{[(\tau+3-x_1)^2+ x_2^2]^2} & \dots & 0\\
\vdots & \vdots & \vdots & \vdots\\
0 & 0 & \dots & 2\varepsilon_0\frac{ (\tau+3-x_1)^2}{[(\tau+3-x_1)^2+ x_N^2]^2}
\end{bmatrix}\\
\text{and}\\
&D= \begin{bmatrix}
  2\varepsilon_0\sum_{j=2}^N\frac{x_j^2}{[(\tau+3-x_1)^2+ x_j^2]^2} & 0 & \dots & 0\\ 
  0 & -2\varepsilon_0\frac{ x_2^2}{[(\tau+3-x_1)^2+ x_2^2]^2} & \dots & 0\\
  \vdots & \vdots & \vdots & \vdots\\
  0 & 0 & \dots & -2\varepsilon_0\frac{ x_N^2}{[(\tau+3-x_1)^2+ x_N^2]^2}
\end{bmatrix}.
\end{aligned}
\end{equation*}
Note that the Matrix $C$ has only one negative eigenvalue while $D$ has only one positive eigenvalue. We use this fact in the following computation.  Taking the above fact into account we get 
\begin{equation*}
\left\{
\begin{aligned}
\mathcal{M}_{\lambda,\Lambda}^+(B)&\leq\mathcal{M}_{\lambda,\Lambda}^+(C)+\mathcal{M}_{\lambda,\Lambda}^+(D)\\&\leq \lambda\left[-2\varepsilon_0\sum_{j=2}^N\frac{ (\tau+3-x_1)^2}{[(\tau+3-x_1)^2+ x_j^2]^2} \right]+\Lambda\left[2\varepsilon_0\sum_{j=2}^N\frac{ (\tau+3-x_1)^2}{[(\tau+3-x_1)^2+ x_j^2]^2} \right]\\
&+\Lambda\left[2\varepsilon_0\sum_{j=2}^N\frac{ x_j^2}{[(\tau+3-x_1)^2+ x_j^2]^2} \right]+\lambda\left[-2\varepsilon_0\sum_{j=2}^N\frac{ x_j^2}{[(\tau+3-x_1)^2+ x_j^2]^2} \right]
\end{aligned}
\right.
\end{equation*}
Rearranging the terms in the above expression we have 
\begin{equation}
\left\{
\begin{aligned}
\mathcal{M}_{\lambda,\Lambda}^{+}(B)&\leq (\Lambda-\lambda)2\varepsilon_0\sum_{j=2}^N\frac{(\tau+3-x_1)^2+ x_j^2}{[(\tau+3-x_1)^2+ x_j^2]^2}\\
&\leq(\Lambda-\lambda)2\varepsilon_0\sum_{j=2}^N\frac{1}{(\tau+3-x_1)^2+ x_j^2}\\
&\leq(\Lambda-\lambda)2\varepsilon_0\frac{(N-1)}{(\tau+3-x_1)^2}.
\end{aligned}
\right.
\end{equation}
In order to compute $\mathcal{M}_{\lambda,\Lambda}^{+}(A)$ we write $A=E+F,$ where 
\begin{equation*}
\begin{aligned}
&E= \begin{bmatrix}
  -\frac{1}{(\tau+3-x_1)^2} & 0 & \dots & 0\\ 
  0 & 0 & \dots & 0\\
  \vdots & \vdots & \vdots & \vdots\\
  
  0 & 0 & \dots & 0
\end{bmatrix}\\
\text{and}\\
&F=\begin{bmatrix}
 0 & 4\varepsilon_0\frac{ (\tau+3-x_1)x_2}{[(\tau+3-x_1)^2+ x_2^2]^2} & \dots &4\varepsilon_0\frac{ (\tau+3-x_1)x_N}{[(\tau+3-x_1)^2+ x_N^2]^2}\\ 
  4\varepsilon_0\frac{ (\tau+3-x_1)x_2}{[(\tau+3-x_1)^2+ x_2^2]^2} & 0 & \dots & 0\\
  \vdots & \vdots & \vdots & \vdots\\
  4\varepsilon_0\frac{ (\tau+3-x_1)x_N}{[(\tau+3-x_1)^2+ x_N^2]^2} & 0 & \dots & 0
\end{bmatrix}.
\end{aligned}
\end{equation*}
Matrix $F$ can be written as $F=F_2+F_3+\dots+F_N,$ where 
\begin{equation*}
\begin{aligned}
&F_2= \begin{bmatrix}
 0 & 4\varepsilon_0\frac{ (\tau+3-x_1)x_2}{[(\tau+3-x_1)^2+ x_2^2]^2} & \dots &0\\ 
  4\varepsilon_0\frac{ (\tau+3-x_1)x_2}{[(\tau+3-x_1)^2+ x_2^2]^2} & 0 & \dots & 0\\
  \vdots & \vdots & \vdots & \vdots\\
  
  0 & 0 & \dots & 0
\end{bmatrix}\\
&F_3= \begin{bmatrix}
 0 &0 & 4\varepsilon_0\frac{ (\tau+3-x_1)x_3}{[(\tau+3-x_1)^2+ x_3^2]^2} & \dots &0\\ 
 0 &0 &0  & \dots &0\\ 
  4\varepsilon_0\frac{ (\tau+3-x_1)x_3}{[(\tau+3-x_1)^2+ x_3^2]^2} & 0  & 0 & \dots & 0\\
  \vdots & \vdots & \vdots & \vdots\\
  0 & 0 & 0 & \dots & 0
\end{bmatrix}\\
&\vdots\\
&F_N= \begin{bmatrix}
 0 &0  & \dots &4\varepsilon_0\frac{ (\tau+3-x_1)x_N}{[(\tau+3-x_1)^2+ x_N^2]^2}\\ 
 0  &0  & \dots &0\\ 
  0  & 0 & \dots & 0\\
  \vdots & \vdots & \vdots & \vdots\\
  4\varepsilon_0\frac{ (\tau+3-x_1)x_N}{[(\tau+3-x_1)^2+ x_N^2]^2} & 0  & \dots & 0
\end{bmatrix}
\end{aligned}
\end{equation*}
Observe that the eigenvalues of $F_i$ are $4\varepsilon_0\frac{ (\tau+3-x_1)x_i}{[(\tau+3-x_1)^2+ x_i^2]^2},$ $-4\varepsilon_0\frac{ (\tau+3-x_1)x_i}{[(\tau+3-x_1)^2+ x_i^2]^2}$ and $0$ with multiplicity $N-2$ for $i=2,3,\dots,N.$ Observe also that
\[\frac{(\tau+3-x_1)^2+x_i^2}{2}\geq(\tau+3-x_1)x_i\]
where we have used the AM-GM inequality with $a=(\tau+3-x_1)^2>0$ and $b= x_i^2\geq0.$ The above inequality further can be rewritten as follows:
\begin{equation}
  \frac{1}{2}\geq\frac{|(\tau+3-x_1)x_i|}{(\tau+3-x_1)^2+ x_i^2}.  
\end{equation}
Now using the subadditive properties of Pucci's operator we have 
\begin{equation*}
\left\{
\begin{aligned}
  \mathcal{M}_{\lambda,\Lambda}^{+}(F)&\leq\mathcal{M}_{\lambda,\Lambda}^{+}(F_{2})+\mathcal{M}_{\lambda,\Lambda}^{+}(F_{3})+\dots+\mathcal{M}_{\lambda,\Lambda}^{+}(F_{N})\\
    &\leq(\Lambda-\lambda)\sum_{j=2}^N\frac{4\varepsilon_0 |(\tau+3-x_1)x_j|}{[(\tau+3-x_1)^2+ x_j^2]^2}\\
    &\leq(\Lambda-\lambda)\sum_{j=2}^N\frac{1}{2}\frac{4\varepsilon_0}{(\tau+3-x_1)^2+ x_j^2}\\
    &\leq(\Lambda-\lambda)\sum_{j=2}^N\frac{2\varepsilon_0}{(\tau+3-x_1)^2}=(\Lambda-\lambda)\frac{2\varepsilon_0(N-1)}{(\tau+3-x_1)^2}.
\end{aligned}
\right.
\end{equation*}
Consequently, we have  
\begin{equation*}
\left\{
\begin{aligned}
  \mathcal{M}_{\lambda,\Lambda}^+(A)&\leq\mathcal{M}_{\lambda,\Lambda}^+(E)+\mathcal{M}_{\lambda,\Lambda}^+(F) \\
  &\leq\lambda\left(\frac{-1}{(\tau+3-x_1)^2}\right)+(\Lambda-\lambda)\frac{2\varepsilon_0(N-1)}{(\tau+3-x_1)^2}
\end{aligned}
\right.
\end{equation*}
Putting all the calculations together we have 
\begin{equation*}
\left\{
\begin{aligned}
\mathcal{M}_{\lambda,\Lambda}^+(D^2g_{\tau}(x))&\leq\mathcal{M}_{\lambda,\Lambda}^+(A)+\mathcal{M}_{\lambda,\Lambda}^+(B)\\
    &\leq\lambda\left(\frac{-1}{(\tau+3-x_1)^2}\right)+(\Lambda-\lambda)\frac{2\varepsilon_0(N-1)}{(\tau+3-x_1)^2}+(\Lambda-\lambda)2\varepsilon_0\frac{(N-1)}{(\tau+3-x_1)^2}\\
    &=\frac{1}{(\tau+3-x_1)^2}\left[-\lambda+4\varepsilon_0(N-1)(\Lambda-\lambda)\right].
\end{aligned}
\right.
\end{equation*}
Choosing $0<\varepsilon_0<\frac{\lambda}{4(N-1)(\Lambda-\lambda)}$ if $\lambda<\Lambda$, we get

  $$\mathcal{M}_{\lambda,\Lambda}^+(D^2g_{\tau}(x))<0.$$  

\end{proof}

\subsection{Starting the moving plane}

Since \(h\) is nondecreasing and \(h(0)=0\),
\begin{equation}\label{eq:h-sign}
 h(t)\le0\quad\text{for}~~t<0,\qquad h(t)>0\quad\text{for}~~t>0.
\end{equation}

\begin{lemma}[Initial-plane lemma]\label{lem:start-plane}
For every \(\tau\le0\),
\[
 w_\tau\ge0\quad\text{in }\Sigma_\tau.
\]
\end{lemma}

\begin{proof}
We first observe that
\begin{equation}\label{eq:negative-set-inequality}
 w_\tau(x)\leq0\quad\Longrightarrow\quad
 \mathcal{M}_{\lambda,\Lambda}^{-}(D^2w_\tau(x))\le0,
 \qquad\tau\le0.
\end{equation}
Indeed, put \(x_1^\tau=2\tau-x_1\).  If \(x_1\le2\tau\), then
\(x_1<0\le x_1^\tau\), and therefore
\[
 h(x_1)f(v)-h(x_1^\tau)f(v_\tau)\le0.
\]
If \(2\tau<x_1<\tau\), then
\(x_1<x_1^\tau<0\).  Since \(w_\tau\leq0\), monotonicity of \(f\) gives
\begin{align*}
 h(x_1)f(v)-h(x_1^\tau)f(v_\tau)
 &\le h(x_1^\tau)\bigl(f(v)-f(v_\tau)\bigr)\le0.
\end{align*}
In either case, \cref{eq:negative-set-inequality} follows from
\cref{eq:reflected-pucci-inequality}.

Suppose now, by contradiction, that
\(\inf_{\Sigma_\tau}w_\tau<0\).  Let \(g_\tau\) be given by
Lemma \ref{lem:pucci-log-barrier} and set \(W_\tau:=w_\tau/g_\tau\).  Since
\(W_\tau(x)\to0\) as \(|x|\to\infty\), then  the negative number
\[
 m:=\inf_{\Sigma_\tau}W_\tau
\]
is attained at some \(x_0\in\Sigma_\tau\).  At \(x_0\),
\(DW_\tau=0\), \(D^2W_\tau\ge0\), and
\[
 D^2w_\tau-mD^2g_\tau=g_\tau D^2W_\tau\ge0.
\]
Using \cref{eq:pucci-difference-minus}, positive homogeneity, and
\cref{eq:negative-set-inequality,eq:barrier-pucci-sign}, we obtain
\begin{align*}
 0
 &\le\mathcal{M}_{\lambda,\Lambda}^{-}(D^2w_\tau-mD^2g_\tau)\\
 &\le\mathcal{M}_{\lambda,\Lambda}^{-}(D^2w_\tau)
       +\Mplus(-mD^2g_\tau)\\
       &=\mathcal{M}_{\lambda,\Lambda}^{-}(D^2w_\tau)
       -m\Mplus(D^2g_\tau)<0
\end{align*}
because \(-m>0\).  This contradiction proves the lemma.
\end{proof}

\subsection{Strict comparison at admissible positions}

Define
\begin{equation}\label{eq:admissible-positions}
 \mathcal A:=\bigl\{\tau\in\R:
 w_\sigma\ge0\text{ in }\Sigma_\sigma
 \text{ for every }\sigma\le\tau\bigr\},
 \qquad
 \tau_*:=\sup\mathcal A.
\end{equation}
By Lemma \ref{lem:start-plane}, \(0\in\mathcal A\), so \(\tau_*\ge0\).

\begin{lemma}[Strictness and the boundary derivative]
\label{lem:continuation}
If either $\tau_{*}<+\infty$ and $\tau\in(-\infty,\tau_{*}]$ or $\tau_{*}=+\infty$ and $\tau\in(-\infty,+\infty),$ then 
\[w_{\tau}(x)>0~~ \text{in}~~ \Sigma_{\tau},~~\frac{\partial w_{\tau}(x)} {\partial x_1}<0~~\text{for}~~x\in T_{\tau}.\]\\
\end{lemma}

\begin{proof}
By definition of $\tau_{*}$ and continuity ( in case $\tau=\tau_{*}<+\infty$ ), we have $w_{\tau}(x)\geq0$ in $\Sigma_{\tau}.$\\
It is enough to show $w_{\tau}$ is never zero in $\Sigma_{\tau}.$ For this 
define the nonnegative, locally bounded difference quotient
\[
 q_\tau(x):=
 \begin{cases}
  \dfrac{f(v_\tau(x))-f(v(x))}{v_\tau(x)-v(x)},&w_\tau(x)>0,\\[7pt]
  f'(v(x)),&w_\tau(x)=0.
 \end{cases}
\]
Since \(h(x_1)\le h(2\tau-x_1)\) in \(\Sigma_\tau\),
\cref{eq:reflected-pucci-inequality} gives
\begin{equation}\label{eq:strictness-linearized}
 \mathcal{M}_{\lambda,\Lambda}^{-}(D^2w_\tau)
 +h(2\tau-x_1)q_\tau(x)w_\tau\le0
 \quad\text{in }\Sigma_\tau.
\end{equation}
Put \(a_\tau(x):=h(2\tau-x_1)q_\tau(x)\).  This coefficient is locally
bounded.  On a bounded subdomain \(D\Subset\Sigma_\tau\), choose
\(c_D\ge\|a_\tau^-\|_{L^\infty(D)}\).  Since \(w_\tau\ge0\),
\cref{eq:strictness-linearized} implies
\[
  \mathcal{M}_{\lambda,\Lambda}^{-}(D^2w_\tau)-c_Dw_\tau
 \le-\bigl(a_\tau+c_D\bigr)w_\tau\le0
 \quad\text{in }D.
\]
Thus \cref{lem:smp-hopf} applies on every such \(D\).  Propagating through
overlapping bounded subdomains of the connected half-space
\(\Sigma_\tau\), we conclude that either \(w_\tau>0\) or
\(w_\tau\equiv0\).

The second alternative is impossible.  Indeed, choose
\(x=(x_1,x')\in\Sigma_\tau\) with
\(x_1<0<2\tau-x_1\).  If \(w_\tau\equiv0\), then by \eqref{eq:reflected-pucci-inequality}, together with \(f(v)>0\), give
$$0=\mathcal{M}_{\lambda,\Lambda}^-(D^2w_{\tau}(x))\leq 
    [h(x_1)-h(2\tau-x_1)]f(v(x))<0,$$

which is a contradiction.  Thus \(w_\tau>0\).

Finally, apply the Hopf boundary lemma locally to
\cref{eq:strictness-linearized}.  Since the outer normal to
\(\Sigma_\tau\) is \(e_1\), we obtain
\(\partial_1w_\tau<0\) on \(T_\tau\).
\end{proof}
\subsection{The moving plane cannot stop}

\begin{lemma}[No finite stopping]\label{lem:no-stopping}
The number \(\tau_*\) defined in \cref{eq:admissible-positions} satisfies
\(\tau_*=+\infty\).
\end{lemma}

\begin{proof}
Suppose, to the contrary, that \(\tau_*<+\infty\).  By continuity,
\(w_{\tau_*}\ge0\), and hence Lemma \ref{lem:continuation} gives
\begin{equation}\label{eq:stopping-strictness}
 w_{\tau_*}>0\text{ in }\Sigma_{\tau_*},
 \qquad
 \partial_1w_{\tau_*}<0\text{ on }T_{\tau_*}.
\end{equation}
From the definition of \(\tau_*\), there are
\(\tau_j>\tau_*\), \(\tau_j\to\tau_*\), such that
\begin{equation}\label{eq:post-stopping-negative}
 \inf_{\Sigma_{\tau_j}}w_{\tau_j}<0.
\end{equation}
Write \(w_j:=w_{\tau_j}\).

\medskip
\noindent\emph{Step 1:}
We claim that, for all sufficiently large \(j\),
\begin{equation}\label{eq:claim-bad-points}
 x\in\Sigma_{\tau_j},\quad w_j(x)\le0
 \quad\Longrightarrow\quad
 \mathcal{M}_{\lambda,\Lambda}^{-}(D^2w_j(x))<0.
\end{equation}
If this were false, after passing to a subsequence we could choose
\(x^j=(s_j,z_j)\in\R\times\R^{N-1}\) such that
\begin{equation}\label{eq:bad-point-sequence}
 w_j(x^j)\le0,\qquad \mathcal{M}_{\lambda,\Lambda}^{-}(D^2w_j(x^j))\ge0.
\end{equation}
Since \(2\tau_j-s_j>\tau_j>0\),
\cref{eq:reflected-pucci-inequality,eq:bad-point-sequence} imply
\(h(s_j)>0\).  Hence
\begin{equation}\label{eq:sj-bounded}
 0<s_j<\tau_j.
\end{equation}

If \((z_j)\) were bounded, a subsequence would satisfy
\(x^j\to x^0\in\overline{\Sigma_{\tau_*}}\).  If
\(x_1^0<\tau_*\), passing to the limit in
\(w_j(x^j)\le0\) would contradict the first part of
\cref{eq:stopping-strictness}.  If \(x_1^0=\tau_*\), the mean-value
theorem applied between \(s_j\) and \(\tau_j\) would give points
\(t_j\in(s_j,\tau_j)\) such that
\(\partial_1w_j(t_j,z_j)\ge0\); passing to the limit would contradict the
second part of \cref{eq:stopping-strictness}.  Therefore,
\begin{equation}\label{eq:tangential-escape}
 |z_j|\longrightarrow+\infty.
\end{equation}

Define the tangential translates
\[
 \bar v_j(x_1,x'):=v(x_1,x'+z_j).
\]
 Boundedness and the local
estimates yield, along a subsequence,
\begin{equation}\label{eq:tangential-limit}
 \bar v_j\longrightarrow\bar v
 \quad\text{in }C^2_{\mathrm{loc}}(\R^N),
\end{equation}

where \(\bar v\ge0\) solves \cref{eq:weighted-entire}. 
By the strong maximum principle, we have either $\bar{v}(x)>0$ in $R^{N}$, or $\bar{v} \equiv 0$. By Lemma \ref{lem:continuation}, we have $\bar{v}_{j}^{\tau_{*}}(x) \geq \bar{v}_{j}(x)$ in $\Sigma_{\tau_{*}}$ which implies
\[\bar{w}_{\tau_{*}}(x):=\bar{v}^{\tau_{*}}(x)-\bar{v}(x) \geq 0, ~~\forall x \in \Sigma_{\tau_{*}}.\]
Now again by repeating the argument as in the proof of Lemma \ref{lem:continuation} can be applied to $\bar{v}$ to conclude that unless $\bar{v} \equiv 0$, we have
\begin{equation}\label{eq4}
\bar{w}_{\tau_{*}}(x)>0 ~~\forall x\in\Sigma_{\tau _{*}},~~\frac{\partial \bar{w}_{\tau_{*}}}{\partial x_{1}}(x)<0 ~~\forall x \in T_{\tau_{*}}. 
\end{equation}
Our next aim is to show that \eqref{eq4} cannot happen and hence $\bar{v} \equiv 0$. In fact, if $\lim _{j \rightarrow+\infty}s_{j}=x_{1}^{0}<\tau_{*}$, then
\[\bar{w}_{\tau_{*}}\left(x_{1}^{0}, 0, \ldots, 0\right)=\lim _{j \rightarrow+\infty} w_{j}\left(x^{j}\right)\leq 0\]
which violates the first part of \eqref{eq4}. If $\lim _{j \rightarrow+\infty} s_{j}=x_{1}^{0}=\tau_{*}$, then due to $s_{j}<\tau_{j}$ and $w_{j}\left(s_{j}, z_{j}\right) \leq 0=w_{j}\left(\tau_{j}, z_{j}\right)$, we can find $t_{j} \in\left(s_{j}, \tau_{j}\right)$ such that $\frac{\partial w_{j}}{\partial x_{1}} \left(t_{j},z_{j}\right) \geq 0.$ It follows that
\[\frac{\partial \bar{w}_{\tau_{*}}}{\partial x_{1}}\left(\tau_{*}, 0\right)\geq 0\]

a contradiction to the second part of \eqref{eq4}.
Consequently,
\begin{equation}\label{eq:tangential-zero-limit}
 \bar v\equiv0,\qquad
 \bar v_j\longrightarrow0\quad\text{in }C^2_{\mathrm{loc}}(\R^N).
\end{equation}

Set \(a_j:=\bar v_j(0)>0\) and \(V_j:=\bar v_j/a_j\).  Then
\begin{equation}\label{eq:normalized-translate-equation}
 -\Mplus(D^2V_j)=c_j(x)V_j,
 \qquad
 c_j(x):=h(x_1)\frac{f(\bar v_j(x))}{\bar v_j(x)}.
\end{equation}
On every compact set, \(c_j\to0\) uniformly.  Harnack inequalities along
a finite chain of balls give local bounds for \(V_j\), because
\(V_j(0)=1\).  Interior compactness therefore gives
\[
 V_j\longrightarrow V\quad\text{in }C^1_{\mathrm{loc}}(\R^N),
 \qquad
 -\Mplus(D^2V)=0,\quad V\ge0,\quad V(0)=1.
\]
For every $R>0,$ the homogeneous Harnack inequality gives 
\[
\sup_{B_R} V\leq C\inf_{B_R}V\leq CV(0)=C,
\]
with $C$ independent of $R.$ Hence $V$ is globally bounded. By Corollary \ref{cor:bounded-homogeneous}, \(V\equiv1\).  It follows that
\begin{equation}\label{eq:log-gradient-vanishes}
 \frac{|D\bar v_j|}{\bar v_j}
 =\frac{|DV_j|}{V_j}\longrightarrow0
 \quad\text{locally uniformly in }\R^N.
\end{equation}
Note also that as $\left\{s_{j}\right\}$ is bounded, the balls $B_{j}:=\left\{x:\left|x-\left(s_{j}, 0\right)\right|<1+2 \tau_{*}\right\},$ for $j=1,2,\ldots$ are contained in some fixed compact subset of $\mathbb{R}^{N}.$ Therefore,
\begin{equation}\label{eq5}
\sup _{x \in B_{j}} \frac{\left|D \bar{v}_{j}(x)\right|}{\bar{v}_{j}(x)} \rightarrow 0 \text { as } j \rightarrow+\infty 
\end{equation}
For $x=(t, 0) \in \mathbb{R} \times \mathbb{R}^{N-1}$ with \(t\in[s_j,\tau_j]\), put \(t^{\tau_j}:=2\tau_j-t\) and define
\begin{equation}\label{eq:ell-j-definition}
 \ell_j(t):=
 h(t)f\bigl(\bar v_j(x)\bigr)
 -h(t^{\tau_j})f\bigl(\bar v_j(x^{\tau_j})\bigr).
\end{equation}
Both \(t\) and \(t^{\tau_j}\) remain in a fixed compact interval of
\((0,+\infty)\), except that \(t\) may approach zero.  Assumption~(L2)
and positivity of \(h'\) on \((0,+\infty)\) imply that, for some
\(\sigma_0>0\),
\begin{equation}\label{eq:h-log-derivative-lower}
 \frac{h'(r)}{h(r)}\ge\sigma_0
 \quad\text{for }0<r\le2\tau_*+2.
\end{equation}
Likewise, (L3)--(L4) imply that, for some \(C_f>0\),
\begin{equation}\label{eq:f-log-derivative-upper}
 0\le\frac{s f'(s)}{f(s)}\le C_f
 \quad\text{for all sufficiently small }s>0.
\end{equation}
A direct differentiation of \cref{eq:ell-j-definition} gives
\begin{equation}\label{eq:ell-j-derivative}
    \left\{
\begin{aligned}
l_{j}^{\prime}(t)= & h(t) f\left(\bar{v}_{j}(x)\right)\left(\frac{h^{\prime}(t)}{h(t)}+\frac{f^{\prime}\left(\bar{v}_{j}(x)\right) \bar{v}_{j}(x)}{f\left(\bar{v}_{j}(x)\right)} \frac{\frac{\partial}{\partial x_{1}} \bar{v}_{j}(x)}{\bar{v}_{j}(x)}\right) \\
&+h\left(t^{\tau_{j}}\right) f\left(\bar{v}_{j}\left(x^{\tau_{j}}\right)\right)\left(\frac{h^{\prime}\left(t^{\tau_{j}}\right)}{h\left(t^{\tau_{j}}\right)}+\frac{f^{\prime}\left(\bar{v}_{j}\left(x^{\tau_{j}}\right)\right) \bar{v}_{j}\left(x^{\tau_{j}}\right)}{f\left(\bar{v}_{j}\left(x^{\tau_{j}}\right)\right)} \frac{\frac{\partial \bar{v}_{j}}{\partial x_{1}}\left(x^{\tau_{j}}\right)}{\bar{v}_{j}\left(x^{\tau_{j}}\right)}\right).
\end{aligned}
\right.
\end{equation}
  By
\cref{eq:tangential-zero-limit,eq:log-gradient-vanishes,%
eq:h-log-derivative-lower,eq:f-log-derivative-upper}, both brackets in
\cref{eq:ell-j-derivative} are positive for all large \(j\), uniformly for
\(t\in[s_j,\tau_j]\).  Hence
\[
 \ell_j(s_j)<\ell_j(\tau_j)=0.
\]
But \cref{eq:reflected-pucci-inequality} evaluated at \(x^j\) yields
\[
 0\le\mathcal{M}_{\lambda,\Lambda}^{-}(D^2w_j(x^j))\le\ell_j(s_j)<0,
\]
contradicting \cref{eq:bad-point-sequence}.  This proves
\cref{eq:claim-bad-points}.

\medskip
\noindent\emph{Step 2: contradiction with the logarithmic barrier.}
For the sequence in \cref{eq:post-stopping-negative}, define
\(W_j:=w_j/g_{\tau_j}\).  As in the proof of
Lemma \ref{lem:start-plane}, \(W_j\) attains a negative minimum
\(m_j\) at some \(y^j\in\Sigma_{\tau_j}\).  At this point
\(w_j(y^j)<0\), so \cref{eq:claim-bad-points} applies.  Therefore,
\begin{align*}
 0
 &\le\Mminus(D^2w_j(y^j)-m_jD^2g_{\tau_j}(y^j))\\
 &\le\Mminus(D^2w_j(y^j))
   +\Mplus(-m_jD^2g_{\tau_j}(y^j))<0,
\end{align*}
where we used \(-m_j>0\) and \cref{eq:barrier-pucci-sign}.  This final
contradiction proves \(\tau_*=+\infty\).
\end{proof}

\subsection{Completion of the proof}

\begin{proof}[Proof of \cref{thm:weighted-liouville}]
If \(v\not\equiv0\), then \cref{eq:zero-or-positive} gives \(v>0\).
By Lemma \ref{lem:no-stopping}, every \(\tau\in\R\) is admissible, and
Lemma \ref{lem:continuation} yields
\[
 v(2\tau-x_1,x')>v(x_1,x')
 \qquad(x_1<\tau).
\]
Given \(s<t\), choose \(\tau=(s+t)/2\).  Then
\[
 v(t,x')>v(s,x')
 \qquad\text{for every }x'\in\R^{N-1},
\]
which proves strict monotonicity in the \(x_1\)-direction.

Assume now that \(h(t)\to+\infty\) as \(t\to+\infty\).  For
\(\sigma>0\), let
$$\Omega_\sigma=B_{1/2}((\sigma+1/2)e_1).$$
where $e_1=(1,0,\dots,0)\in\mathbb{R}^N.$
Then the first principle eigenvalue $\mu_1^+$ of the problem
\begin{equation*}\label{eq6}
\left\{
\begin{aligned}
-\mathcal{M}_{\lambda,\Lambda}^+(D^2v)&=\mu v~~\text{in}~~\Omega_\sigma,\\
v&=0~~\text{on}~~\partial\Omega_\sigma
\end{aligned}
\right.
\end{equation*}
is independent of $\sigma.$ If $v$ is a positive solution of \cref{eq:weighted-entire}, then $v$ is strictly increasing in $x_1,$ consequently we have 
\begin{equation}\label{sigma1}
     -\mathcal{M}_{\lambda,\Lambda}^+(D^2v)=h(x_1)f(v)\geq h(\sigma)m_0v,~~\forall x\in\Omega_\sigma
\end{equation}

where 
$$m_0=\frac{f(\xi_0)}{\text{sup}_{\mathbb{R}^N}v}>0,~~\xi_0=\displaystyle{\min_{x_1=0,|x|\leq1}}v(x)>0.$$
Since $v>0$ in $\mathbb{R}^N,$ by \eqref{sigma1} $\mu_1^+\geq h(\sigma)m_0.$ It follows that 
\begin{align*}
\overline{\lim\limits}_{\sigma\to+\infty}h(\sigma)\leq\frac{\mu_1^+}{m_0}<+\infty,
\end{align*}
which is a contradiction to our assumption that $h(t)\to+\infty$ as $t\to+\infty.$ Therefore, \cref{eq:weighted-entire} has only the trivial solution $v\equiv0$.This completes the proof.\\

\end{proof}
\section{A Priori Estimates and Blow-up Analysis}
\label{sec:reduction}
\subsection{Spectral preliminaries and localization of maximum points}
\begin{remark}[Eigenvalue and maximum principle]\label{rem:2.2}
(1) We know that if there exists a positive supersolution to
\[
-\mathcal{M}_{\lambda,\Lambda}^{+}(D^2u)-\mu u
\]
then maximum principle holds, see Theorem 3.3 \cite{armstrong2009principal}. In particular,
if $\mu_1^+(G)$ is positive then for any $\mu<\mu_1^+(G)$ corresponding
eigenfunction $\vartheta_1^+$ serve as the existence of positive
eigenfunction for the considered operator. Thus, in this case, even if
the operator
\[
-\mathcal{M}_{\lambda,\Lambda}^{+}(D^2u)-\mu u
\]
is not proper but the maximum principle holds. Moreover the strong
maximum principle also holds for the same operator. Indeed, for
every $x_0\in G$ and every sufficiently small $r>0$,
$B_r(x_0)\Subset G$, and the scaling property of the first
half-eigenvalue gives
\[
\mu_1^+(B_r(x_0))
=
r^{-2}\mu_1^+(B_1)
\longrightarrow +\infty
\quad\text{as }r\to0.
\]
Hence, for every fixed $\mu\in\mathbb{R}$, one can choose $r>0$ so
small that
\[
\mu<\mu_1^+(B_r(x_0)).
\]
The maximum principle is, therefore, valid in $B_r(x_0)$, irrespective
of the value of $\mu$. Consequently, if $u\geq0$ is a supersolution of
\[
-\mathcal{M}_{\lambda,\Lambda}^{+}(D^2u)-\mu u\geq0
\quad\text{in }G
\]
and $u(x_0)=0$ at some interior point $x_0\in G$, the local maximum
principle implies that $u\equiv0$ in a neighborhood of $x_0$. The
usual propagation argument, together with the connectedness of $G$,
yields $u\equiv0$ in $G$.

(2) From the characterization of principle eigenvalue, we observe that if
$G\subset\mathbb{R}^N$ be a bounded regular domain and
\[
-\mathcal{M}_{\lambda,\Lambda}^{+}(D^2w)
\geq \;(\text{resp. }>)\; \mu w
\quad\text{in }G,
\]
for some positive continuous function $w$, then we have
\[
\mu\leq\mu_1^+
\quad
(\text{resp. }\mu<\mu_1^+).
\]
\end{remark}

\begin{lemma}[Positive solution with positive boundary values]
\label{lem:2.3}
Let $G\subset\mathbb{R}^N$ be a bounded regular connected domain, let $w\in C(\overline{G})$
$w>0$ in $\overline{G}$ and
\[
-\mathcal{M}_{\lambda,\Lambda}^{+}(D^2w)=\mu w
\quad\text{in }G.
\]
Then
\[
\mu<\mu_1^+(G).
\]
\end{lemma}

\begin{proof}
We prove this Lemma by method of contradiction and assume
\[
\mu\geq\mu_1^+(G).
\]
In order to get contradiction we use sliding argument. Let us define
\[
t^*
:=
\sup\left\{
t>0:t\phi_1^+\leq w\text{ in }G
\right\}.
\]
As $w$ and $\phi_1^+$ both are continuous function in $G$,
therefore $t^*$ can not be infinite. Moreover, $t^*>0$. Indeed, as
$w>0$ in $G$ and $\phi_1^+=0$ on $\partial G$, so
\[
\frac{w(x)}{\phi_1^+(x)}
\longrightarrow+\infty
\quad\text{as }x\to\partial G.
\]
Thus there exists some point $x^*\in G$ such that
\[
w(x^*)=t^*\phi_1^+(x^*)
\]
and
\[
w(x)\geq t^*\phi_1^+(x)
\quad\text{for all }x\in G.
\]
Notice that $\mu\geq\mu_1^+$ is not possible as, in this case, $w$ is
supersolution and $t^*\phi_1^+$ is solution of
\[
-\mathcal{M}_{\lambda,\Lambda}^{+}(D^2v)
-\mu_1^+(G)v=0
\quad\text{in }G
\]
and
\[
w\geq t^*\phi_1^+
\]
with equality at $x^*$. So strong maximum principle gives
\[
w=t^*\phi_1^+
\]
which is not possible as $w>0$ on $\partial G$ where as
$\phi_1^+=0$ on the boundary. Thus
\[
\mu<\mu_1^+.
\]
\end{proof}

\begin{proof}[Proof of Proposition \ref{prop:necessary-spectral-restriction}]

We prove separately the two inequalities.

\medskip
\noindent
\textbf{Step 1: The region $G^+$.}

By construction,
\[
a(x)\geq0
\quad\text{for }x\in G^+.
\]
Since $u>0$ in $\Omega$, \cref{eq:problem1} gives
\[
-\mathcal{M}_{\lambda,\Lambda}^{+}(D^2u)
=
\mu u+a(x)u^p
\geq\mu u
\quad\text{in }G^+.
\]
Moreover, on the nonempty open set $\Omega^+$ we have $a(x)>0$.
Hence
\[
-\mathcal{M}_{\lambda,\Lambda}^{+}(D^2u)
>
\mu u
\quad\text{in }\Omega^+.
\]
Thus $u$ is a positive supersolution of
\[
-\mathcal{M}_{\lambda,\Lambda}^{+}(D^2v)
=
\mu v
\]
in $G^+$, and the inequality is strict on a nonempty open subset.
By Lemma \ref{lem:2.3},
\begin{equation}\label{mugplus}
\mu<\mu_1^+(G^+).
\end{equation}

\medskip
\noindent
\textbf{Step 2: The region $G^-=Z_-$.}

Notice that $G^-$ is contained in the zero set $\Omega^0$.
Consequently,
\[
a(x)=0
\quad\text{for all }x\in G^-.
\]
Therefore, \cref{eq:problem1} reduces in $G^-$ to
\begin{equation}\label{eq:3.2}
-\mathcal{M}_{\lambda,\Lambda}^{+}(D^2u)
=
\mu u
\quad\text{in }G^-.
\end{equation}
Furthermore, since $u>0$ in $\Omega$ and
$u\in C(\overline{\Omega})$, we have
\[
u>0\quad\text{on }\overline{G^-}.
\]
In particular,
\[
u>0\quad\text{on }\partial Z_-.
\]
Applying \cref{lem:2.3} to \cref{eq:3.2} with $G=Z^o_-$ yields
\begin{equation}\label{eq:3.3}
\mu<\mu_1^+(Z^o_-).
\end{equation}
Combining \cref{eq:3.3} and \cref{mugplus}, we obtain
\[
\mu
<
\min\{\mu_1^+(G^+),\mu_1^+(G^-)\}
=
\sigma_*.
\]
This proves the lemma.
\end{proof}
\begin{lemma}[Positive barriers near neutral components]
\label{lem:positive-barrier}
Let
\[
K_1,\ldots,K_m\Subset \Omega
\]
be the mutually disjoint thick neutral components, with connected
$C^2$ interiors. Set $B=\sup\Ipar<\sigma_*,$ where $\Ipar$ is a bounded interval in $\mathbb{R}.$ Assume that
\[
B<\min_{1\le j\le m}\mu_1^+(K_j^\circ).
\]
Then there exists a function
\[
\psi\in C^2(\Omega), \qquad \min_{\Omega}\psi>0,
\]
such that
\[
-\mathcal{M}_{\lambda,\Lambda}^{+}(D^2\psi)>B\psi
\]
in a neighborhood of each $K_j$.
\end{lemma}

\begin{proof}
Since
\[
B<\mu_1^+(K_j^\circ)
\qquad\text{for every }j=1,\ldots,m,
\]
by the continuity of the principal half-eigenvalue with respect to
smooth exterior domain approximation, we may choose mutually disjoint
smooth domains $D_j$ satisfying
\[
K_j\Subset D_j\Subset\Omega
\]
and
\[
\nu_j:=\mu_1^+(D_j)>B.
\]
Indeed, if $D_j^{(n)}$ is a decreasing sequence of smooth exterior
neighborhoods of $K_j$ such that
\[
D_j^{(n)}\downarrow K_j^\circ,
\]
then
\[
\mu_1^+(D_j^{(n)})
   \longrightarrow
\mu_1^+(K_j^\circ).
\]
Hence, for $n$ sufficiently large,
\[
B<\mu_1^+(D_j^{(n)}).
\]

Let $\phi_j>0$ be the positive principal eigenfunction corresponding
to $\nu_j$, normalized, for instance, by
\[
\|\phi_j\|_{L^\infty(D_j)}=1.
\]
Thus
\[
\begin{cases}
-\mathcal{M}_{\lambda,\Lambda}^{+}(D^2\phi_j)
   =\nu_j\phi_j
   & \text{in }D_j,\\[1mm]
\phi_j>0
   & \text{in }D_j,\\[1mm]
\phi_j=0
   & \text{on }\partial D_j.
\end{cases}
\]

Since $K_j\Subset D_j$, set
\[
\delta_j:=\operatorname{dist}(K_j,\partial D_j)>0.
\]
Choose $\varepsilon_j>0$ sufficiently small so that
\[
0<2\varepsilon_j<\delta_j.
\]
We may then choose
\[
\chi_j\in C_c^\infty(D_j),
\qquad
0\le \chi_j\le 1,
\]
such that
\[
\chi_j\equiv1
\quad\text{on }
N_{\varepsilon_j}(K_j)
:=
\bigl\{
x\in\Omega:
\operatorname{dist}(x,K_j)<\varepsilon_j
\bigr\},
\]
and
\[
\chi_j\equiv0
\quad\text{outside }
N_{2\varepsilon_j}(K_j).
\]
The supports of the $\chi_j$ can be chosen mutually disjoint because
the sets $K_j$ are mutually disjoint and compactly contained in
$\Omega$.

Fix any constant $c>0$ and define
\[
\psi(x)
:=
\sum_{j=1}^m\chi_j(x)\phi_j(x)
+
c\left(
1-\sum_{j=1}^m\chi_j(x)
\right).
\]
Equivalently,
\[
\psi(x)
=
c+
\sum_{j=1}^m
\chi_j(x)\bigl(\phi_j(x)-c\bigr).
\]
Since the supports of the functions $\chi_j$ are mutually disjoint,
at every point at most one cutoff is nonzero. Hence $\psi$ is a
$C^2$ function on $\Omega$.

Moreover, $\psi$ is strictly positive. Indeed, on the support of
$\chi_j$ we have
\[
\psi
=
\chi_j\phi_j+(1-\chi_j)c,
\]
which is a convex combination of the two positive quantities
$\phi_j$ and $c$. Outside all the cutoff supports,
\[
\psi\equiv c.
\]
Since
\[
\overline{N_{2\varepsilon_j}(K_j)}
\Subset D_j
\]
and $\phi_j>0$ in $D_j$, we have
\[
\min_{\overline{N_{2\varepsilon_j}(K_j)}}\phi_j>0.
\]
Consequently,
\[
\min_{\Omega}\psi>0.
\]

Finally, in the neighborhood
$N_{\varepsilon_j}(K_j)$ we have
\[
\chi_j\equiv1,
\qquad
\chi_i\equiv0
\quad(i\neq j).
\]
Therefore
\[
\psi=\phi_j
\qquad\text{in }N_{\varepsilon_j}(K_j).
\]
In particular,
\[
D^2\psi=D^2\phi_j
\qquad\text{in }N_{\varepsilon_j}(K_j),
\]
and hence
\[
-\mathcal{M}_{\lambda,\Lambda}^{+}(D^2\psi)
=
-\mathcal{M}_{\lambda,\Lambda}^{+}(D^2\phi_j)
=
\nu_j\phi_j.
\]
Since $\nu_j>B$ and $\psi=\phi_j$ in this neighborhood,
\[
-\mathcal{M}_{\lambda,\Lambda}^{+}(D^2\psi)
=
\nu_j\psi
>
B\psi.
\]
Thus the desired strict inequality holds in a neighborhood of every
$K_j$.
\end{proof}
\begin{lemma}\label{lem:2.5}[Ratio localization and preserved maximum]
Suppose that $\mu_n \leq B$ and that $\psi$ is the barrier given by Lemma \ref{lem:positive-barrier} For solutions of
\[
-\mathcal{M}_{\lambda,\Lambda}^{+}(D^2u)
=
\mu u+a(x)u^p
\quad \text{in } \Omega,
\]
choose $x_n$ maximizing $u_n/\psi$ and put
\[
M_n=u_n(x_n).
\]
Then $x_n$ lies outside every neutral interior and
\begin{equation}\label{max:inequality}
   \frac{\min_{\overline{\Omega}}\psi}
     {\max_{\overline{\Omega}}\psi}
\|u_n\|_{\infty}
\leq
M_n
\leq
\|u_n\|_{\infty}. 
\end{equation}

For every scale $r_n\to 0$, the rescaled functions
\[
v_n(y)
=
\frac{u_n(x_n+r_n y)}{M_n}
\]
have a fixed global bound, and every locally uniform limit on a domain
containing the  origin satisfies
\[
0\leq v\leq 1,
\qquad
v(0)=1.
\]
\end{lemma}

\begin{proof}
The ratio $u_n/\psi$ has a positive interior maximum because the
boundary values of $u_n$ are zero and $\psi$ is strictly positive on
$\overline{\Omega}$. Write
\[
s_n=\frac{M_n}{\psi(x_n)}.
\]
At the maximum point $x_n$, the function $s_n\psi$ is an upper test
function for $u_n$.

Suppose that $x_n$ belongs to the interior of a neutral component.
Then
\[
a(x_n)=0
\]
Hence, by the viscosity subsolution inequality,
\[
s_n
\left[
-\mathcal{M}_{\lambda,\Lambda}^{+}
\bigl(D^2\psi(x_n)\bigr)
\right]
\leq
\mu_n s_n\psi(x_n)
\leq
B s_n\psi(x_n),
\]
which contradicts the strict barrier inequality
\[
-\mathcal{M}_{\lambda,\Lambda}^{+}(D^2\psi)
>
B\psi
\]
near every thick neutral component. Therefore, $x_n$ lies outside every
neutral interior.

Let $y_n$ be an ordinary maximum point of $u_n$. Since $x_n$ maximizes
$u_n/\psi$, we have
\[
\frac{M_n}{\psi(x_n)}
\geq
\frac{\|u_n\|_{\infty}}{\psi(y_n)}.
\]
Consequently,
\[
M_n
\geq
\frac{\psi(x_n)}{\psi(y_n)}
\|u_n\|_{\infty}
\geq
\frac{\min_{\overline{\Omega}}\psi}
     {\max_{\overline{\Omega}}\psi}
\|u_n\|_{\infty}.
\]
The upper bound
\[
M_n\leq \|u_n\|_{\infty}
\]
is immediate.

Moreover, by the maximality of $u_n/\psi$,
\[
\frac{u_n(x_n+r_n y)}
     {\psi(x_n+r_n y)}
\leq
\frac{u_n(x_n)}{\psi(x_n)}
=
\frac{M_n}{\psi(x_n)}.
\]
Therefore,
\[
0\leq v_n(y)
=
\frac{u_n(x_n+r_n y)}{M_n}
\leq
\frac{\psi(x_n+r_n y)}{\psi(x_n)}
\leq
\frac{\max_{\overline{\Omega}}\psi}
     {\min_{\overline{\Omega}}\psi}.
\]
Thus the sequence $(v_n)$ has a fixed global bound.

Since $\psi$ is uniformly continuous on $\overline{\Omega}$ and
$\min_{\overline{\Omega}}\psi>0$, for $r_n\to 0$ we have, locally
uniformly in $y$,
\[
\frac{\psi(x_n+r_n y)}{\psi(x_n)}
\longrightarrow 1.
\]
Hence every locally uniform limit $v$ satisfies
\[
0\leq v\leq 1.
\]
Finally,
\[
v_n(0)
=
\frac{u_n(x_n)}{M_n}
=
1,
\]
and therefore
\[
v(0)=1.
\]
\end{proof}




\begin{lemma}\label{scaled:solution}
    Let $\Omega\subset\mathbb{R}^N$ be a bounded domain  with $C^2$ boundary.Let $x_n$ be as in Lemma \ref{lem:2.5} and $u_n$ satisfies \eqref{eq:problem1} with $M_n=u_n(x_n)\to\infty$ as $n\to\infty.$ Consider $y=\frac{x-x_n}{\rho_n}$ where $\rho_n^\frac{2+\gamma}{p-1}M_n=1$. Define $v_n(y)=\rho_n^\frac{2+\gamma}{p-1}u_n(x)$ where $u_n(x)$ is solution of \eqref{eq:problem1}. Then for $n$ large enough, the function $v_n(y)$ is a viscosity solution of 
    \begin{equation}\label{eq:rescaled solution}
        -\mathcal{M}^{+}_{\lambda,\Lambda}(D^2v_n)=\rho_n^2\mu_nv_n+a_n(y)v_n^p~~~~\text{in}~~~\Omega_n
    \end{equation}
    where $a_n(y)=a(x_n+\rho_ny)\rho_n^{-\gamma}$ and
    $\Omega_n=\{y\in\mathbb{R}^N:x_n+\rho_ny\in\Omega\}.$
\end{lemma}
\begin{proof}
  We begin by verifying that \(v_n\) satisfies the viscosity subsolution condition for \eqref{eq:rescaled solution}.\\
 Let \(\phi\in C^2\) and \(y_0\in \Omega_n\) be such that \(v_n-\phi\) attains a local maximum at \(y_0\). Without loss of generality, we may assume that \(\phi\) touches \(v_n\) from above at \(y_0\); that is,
  \[
  (v_n-\phi)(y)\leq(v_n-\phi)(y_0)=0
  \]
  for all $y$ in the neighborhood of $y_0.$ Define 
  \[
  \Psi(x)=\rho_n^{-\big(\frac{2+\gamma}{p-1}\big)}\phi\Big(\frac{x-x_n}{\rho_n}\Big)
  \]
  Then $\Psi$ touches $u_n$ from the above at $x_0=x_n+\rho_ny_0\in\Omega,$ 
  \[
  u_n(x_0)=u_n(x_n+\rho_ny_0)=\rho_n^{-\big(\frac{2+\gamma}{p-1}\big)}v_n(y_0)=\rho_n^{-\big(\frac{2+\gamma}{p-1}\big)}\phi(y_0)=\Psi(x_0),
\]
and
\[
  u_n(x)=u_n(x_n+\rho_ny)=\rho_n^{-\big(\frac{2+\gamma}{p-1}\big)}v_n(y)\leq \rho_n^{-\big(\frac{2+\gamma}{p-1}\big)}\phi(y)=\Psi(x),
  \]
  for all $x$ in a neighborhood of $x_0.$\\
 We next express the derivatives of \(\Psi(x)\) in terms of the corresponding derivatives of \(\phi(y)\)
 \[
 D_x\Psi(x)=\rho_n^{-\big(\frac{2+\gamma}{p-1}\big)-1}D_y\phi\Big(\frac{x-x_n}{\rho_n}\Big),
 \]
 \[
 D^2_x\Psi(x_0)=\rho_n^{-\big(\frac{2+\gamma}{p-1}\big)-2}D^2_y\phi(y_0)
 \]
 Since $u_n$ is a viscosity subsolution of \eqref{eq:problem1}, we have 
 \begin{equation*}
     \begin{aligned}
         -\mathcal{M}_{\lambda,\Lambda}^{+}(D^2\Psi(x_0))&\leq\mu_nu_n(x_0)+a(x_0)u_n^p(x_0)\\
         -\mathcal{M}_{\lambda,\Lambda}^{+}(\rho_n^{-\big(\frac{2+\gamma}{p-1}\big)-2}D^2_y\phi(y_0))&\leq\mu_nv_n(y_0)\rho_n^{-\big(\frac{2+\gamma}{p-1}\big)}+a(x_n+\rho_ny_0)\rho_n^{-\big(\frac{2+\gamma}{p-1}\big)p}v_n^p(y_0)\\
         -\mathcal{M}_{\lambda,\Lambda}^{+}(D^2_y\phi(y_0))&\leq\rho_n^2\mu_nv_n(y_0)+a(x_n+\rho_ny_0)\rho_n^{-\gamma}v_n^p
     \end{aligned}
 \end{equation*}
  Hence, $v_n$ is a subsolution of \eqref{eq:rescaled solution}. The supersolution case is analogous.
\end{proof}
\begin{lemma}\label{scaled:solution2}
    Let $v_n$ be as defined in the above \cref{scaled:solution} which satisfy \eqref{eq:rescaled solution}. Let hypotheses of Lemma \ref{scaled:solution} hold. Let $\delta_n=d(x_n,\partial \Omega) $.  Consider $z=\beta_n^{-1}y$ where $\beta_n=(\rho_n/\delta_n)^{\gamma/2}$. Define $w_n(z)=v_n(\beta_nz)$. Then for $n$ large enough, the function $w_n(z)$ is a viscosity solution of 
    \begin{equation}\label{eq:rescaled solution2}
        -\mathcal{M}^{+}_{\lambda,\Lambda}(D^2w_n)=(\beta_n\rho_n)^2\mu_nw_n+\beta_n^2a_n(\beta_nz)w_n^p~~~~\text{in}~~~\widetilde{\Omega}_n
    \end{equation}
    where 
    $\widetilde{\Omega}_n=\beta_n^{-1}\Omega_n.$
\end{lemma}
\begin{proof}
     We begin by verifying that \(w_n\) satisfies the viscosity subsolution condition for \eqref{eq:rescaled solution2}.\\
 Let \(\phi\in C^2\) and \(z_0\in \Omega_n\) be such that \(w_n-\phi\) attains a local maximum at \(z_0\). Without loss of generality, we may assume that \(\phi\) touches \(w_n\) from above at \(z_0\); that is,
  \[
  (w_n-\phi)(z)\leq(w_n-\phi)(z_0)=0
  \]
  for all $z$ in the neighborhood of $z_0.$ Define 
  \[
  \Psi(y)=\phi(\beta_n^{-1}y)
  \]
  Then $\Psi$ touches $v_n$ from the above at $y_0=\beta_nz_0\in\Omega,$ 
  \[
  v_n(y_0)=v_n(\beta_nz_0)=w_n(z_0)=\phi(z_0)=\Psi(y_0),
\]
and
\[
  v_n(y)=v_n(\beta_nz)=w_n(z)\leq \phi(z)=\Psi(y),
  \]
  for all $y$ in a neighborhood of $y_0.$\\
   We next express the derivatives of \(\Psi(y)\) in terms of the corresponding derivatives of \(\phi(z)\)
 \[
 D_y\Psi(y)=\beta_n^{-1}D_z\phi(z),
 \]
 \[
 D^2_y\Psi(y_0)=\beta_n^{-2}D^2_z\phi(z_0)
 \]
 Since $v_n$ is a viscosity subsolution of \eqref{eq:rescaled solution}, we have 
 \begin{equation*}
     \begin{aligned}
         -\mathcal{M}_{\lambda,\Lambda}^{+}(D_y^2\Psi(y_0))&\leq\rho_n^2\mu_nv_n(y_0)+a_n(y_0)v_n^p(y_0)\\
         -\mathcal{M}_{\lambda,\Lambda}^{+}(\beta_n^{-2}D^2_z\phi(z_0))&\leq\rho_n^2\mu_nw_n(z_0)+a_n(\beta_nz_0)w_n^p(z_0)\\
         -\mathcal{M}_{\lambda,\Lambda}^{+}(D^2_z\phi(z_0))&\leq(\beta_n\rho_n)^2\mu_nw_n(z_0)+a_n(\beta_nz_0)\beta_n^2w_n^p(z_0)
     \end{aligned}
 \end{equation*}
  Hence, $w_n$ is a subsolution of \eqref{eq:rescaled solution2}. The supersolution case is analogous.
\end{proof}
\subsection{Exclusion of blow-up at the interfaces}

Let us set
\begin{equation}\label{eq:safe-exponent}
 \psafe:=
 \begin{cases}
   \dfrac{\NtildeP}{\NtildeP-2}, & \NtildeP>2,\\[6pt]
   +\infty, & \NtildeP\le2.
 \end{cases}
~~~~~~~~~ \text{where}~~~~\NtildeP=\frac{\lambda}{\Lambda}(N-1)+1
\end{equation}
\begin{lemma}\label{lemma3.3}
Let $\{x_n\}$ be the sequence defined in Lemma \eqref{lem:2.5} and  $u_n(x_n)\to\infty$ as  $n\to\infty.$ Assume that $p\in(1,\psafe]~~\text{if}~~\NtildeP>2$ and $1<p<\infty~~\text{if}~~\NtildeP\le2$, where $\psafe$ is given by \eqref{eq:safe-exponent}. Suppose further that assumption \ref{ass:geometry} (G2) holds and that $a(x)=\alpha(x)\,[d(x,\Gamma_{+})]^{\gamma}$ for $x\in\Omega^{+}$ near $\Gamma_{+},$ where $\alpha$ is a positive continuous function defined in a neighbourhood of $\Gamma_{+}$ and $\gamma>0.$ Then $\Gamma_{+}$ contains no limit point of the sequence $\{x_n\}.$
\end{lemma}
\begin{proof}
Assume by contradiction that there exists a subsequence such that $x_n\to x_0\in\Gamma_+.$ As $\Gamma_+$ is at the positive distance from $\overline{\Omega^{-}}$ and Lemma \eqref{lem:2.5} ensures that $x_{n}\not\in\text{int}(Z^{+}_{-})$ therefore if necessary after passing to a subsequence we may assume that $\{x_n\}\subset\overline{\Omega^+}.$ Let us set 
\[M_n:=u_n(x_n),~~\rho_n:=M_n^\frac{-(p-1)}{(2+\gamma)},\] 
and observe that $M_n\to\infty$ and $\rho_n\to0$ as $n\to\infty.$ Following the standard approach in scaling and blow-up technique let us define the following rescaled function
\[v_n(y):=\frac{u_n(x_{n}+\rho_{n}y)}{M_n},\]
By Lemma \ref{scaled:solution}, $v_{n}$ satisfies 
\begin{equation}\label{change}
-\mathcal{M}^{+}_{\lambda,\Lambda}(D^2v_n)=\rho_n^2\mu_nv_n+a_n(y)v_n^p~~\text{in}~~\Omega_{n}
\end{equation}
where 
\[a_n(y)=a(x_n+\rho_ny)\rho_n^{-\gamma}~~\text{and}~~\Omega_{n}=\{y\in\mathbb{R}^{N}~|~x_{n}+\rho_{n}y\in\Omega\}.\]
Let $\delta_n=d(x_n,\Gamma_+)=d(x_n,z_n),$ where $z_n\in\Gamma_+.$ We may assume that $z_n\to z_0.$ Denote by $\nu_n$ the unit normal of $\Gamma_+$ at $z_n$ pointing inward of $\Omega^+.$ We obtain
$d_{\Gamma_+}(x_n+\rho_ny)=|\delta_n+\rho_n\nu_n.y+o(\rho_n)|$ and for large $n,$
\begin{equation}\label{eq4.3}
a_n(y)=\alpha(x_n+\rho_ny)\Big[\Big(\frac{\delta_n}{\rho_n}\Big)+\nu_n.y+o(1)\Big]^\gamma~~~\text{as long as}~~x_n+\rho_ny\in\Omega^+.
\end{equation}
    (i) Suppose that $\left\{\frac{\delta_n}{\rho_n}\right\}$ is bounded from above. By passing to a suitable subsequence, we may assume that $\frac{\delta_n}{\rho_n}\to s\ge0~~\text{as}~~n\to\infty.$ Then by \eqref{eq4.3}, $a_n(y)\to h(s+\nu_0.y)~~\text{as}~~n\to\infty,$ where $h(t)=\alpha(x_0)(t^+)^\gamma.$\\
Now, by applying Proposition \ref{lem:compactness}, we obtain that 
    $$\lim\limits_{n\to\infty}v_n=v,~~\text{uniformly in}~~B_R(0)~\forall ~R>0~\text{and}~v(0)=1.$$
    and then from \eqref{change}, Using stability of viscosity solutions we obtain
    \begin{equation*}
    \left\{
    \begin{aligned}
-\mathcal{M}^{+}_{\lambda,\Lambda}(D^2v)&=h(s+\nu_0.y)v^p~~\text{in}~~\mathbb{R}^N\\
v&\ge0~~\text{in}~~\mathbb{R}^N\\
v(0)&=\max_{\mathbb{R}^{N}}{v}=1.
\end{aligned}
\right.
\end{equation*}
By a translation and rotation of the coordinates we may assume that $s+\nu_0.y=y_1$ and hence $v$ is a solution of
\begin{equation}\label{refint}
\left\{
\begin{aligned}
-\mathcal{M}^{+}_{\lambda,\Lambda}(D^2v)&=h(y_1)v^p~~\text{in}~~\mathbb{R}^N\\
v&\ge0~~\text{in}~~\mathbb{R}^N\\
\max_{\mathbb{R}^{N}}{v}&=1.
\end{aligned}
\right.
\end{equation}
This contradicts \cref{thm:weighted-liouville}.\\
(ii). Suppose that $\left\{\frac{\delta_n}{\rho_n}\right\}$ is unbounded from above. By selecting a suitable subsequence, we may assume that $\beta_n:=(\rho_n/\delta_n)^{\gamma/2}\to0$ as $n\to\infty.$ By introducing the variable $z:=\beta_{n}^{-1}y$ and $w_n(z):=v_n(\beta_nz).$ By Lemma \ref{scaled:solution2}, $w_n$ satisfies
\begin{equation*}
-\mathcal{M}^{+}_{\lambda,\Lambda}(D^2w_n)=(\beta_n\rho_n)^2\mu_nw_n+\beta_n^2a_n(\beta_nz)w_n^p.
\end{equation*}
Observe that
\begin{equation*}
\beta_n^2a_n(\beta_nz)=\alpha(x_n+\beta_n\rho_nz)[1+\beta_n^{1+2/\gamma}\nu_n.z+o(\beta_n)]^\gamma\to\alpha(x_0)~~\text{as}~~n\to\infty.
\end{equation*}
It follows that the problem 
\begin{equation*}
\left\{
\begin{aligned}
-\mathcal{M}^{+}_{\lambda,\Lambda}(D^2w)&=\alpha(x_0)w^p~~\text{in}~~\mathbb{R}^N\\
w&\ge0~~\text{in}~~\mathbb{R}^N\\
w(0)&=1
\end{aligned}
\right.
\end{equation*}
has a solution, contradicting Theorem 4.1\cite{CutriLeoni2000}.\\

\end{proof}
\begin{lemma}
Let $\{x_n\}$ be the sequence defined in Lemma \eqref{lem:2.5} and suppose that $u_n(x_n)\to\infty$  as  $n\to\infty.$ Assume that $p\in(1,\psafe]~~\text{if}~~\NtildeP>2$ and $1<p<\infty~~\text{if}~~\NtildeP\le2$, where $\psafe$ is given by \eqref{eq:safe-exponent}. Assume also that  assumption \ref{ass:geometry} (G1) holds and 
\begin{align*}
a(x)&=\alpha_1(x)\,[d(x,\Gamma)]^{\gamma_2}
\quad \text{for } x\in\Omega^{+} \text{ near } \Gamma,\\
a(x)&=-\alpha_2(x)\,[d(x,\Gamma)]^{\gamma_2}
\quad \text{for } x\in\Omega^{-} \text{ near } \Gamma,
\end{align*}
where $\alpha_1(x)$ and $\alpha_2(x)$ are positive continuous functions defined in a neighbourhood of $\Gamma$ and $\gamma_2>0.$ Then $\Gamma$ contains no limit point of the sequence $\{x_n\}.$
\end{lemma}
\begin{proof}
Suppose by contradiction that $x_n\to x_0\in\Gamma$ (if necessary we can take subsequence). By passing to a further subsequence if necessary we may assume that $\{x_n\}\in\overline{\Omega^-}$ or $\{x_n\}\in\overline{\Omega^+}.$ For definiteness we assume that the first case occurs. As in the Lemma \ref{lemma3.3}, let us set 
\[M_n:=u_n(x_n)~~\text{and}~~\rho_n:=M_n^\frac{-(p-1)}{(2+\gamma_2)}.\] We find that $M_n\to\infty$ and $\rho_n\to0$ as $n\to\infty.$ Let us also set 
\begin{align*}
\Omega_{n}&=\{y\in\mathbb{R}^{N}~|~x_{n}+\rho_{n}y\in\Omega\}\\
\Omega_{n}^{+}&=\{y\in\mathbb{R}^{N}~|~x_{n}+\rho_{n}y\in\Omega^{+}\}\\
\Omega_{n}^{-}&=\{y\in\mathbb{R}^{N}~|~x_{n}+\rho_{n}y\in\Omega^{-}\}.
\end{align*}
Note that as $u_{n}$ solves \eqref{eq:problem1} with $\mu=\mu_{n}$ we find that the rescaled function \[v_n(y):=\frac{u_n(x_{n}+\rho_{n}y)}{M_n}\] 
satisfies the following equation
\begin{equation}\label{eq4.41}
-\mathcal{M}^{+}_{\lambda,\Lambda}(D^2v_n)=\rho_n^2\mu_nv_n+a_n(y)v_n^p~~\text{in}~~\Omega_{n},
\end{equation}
where $a_n(y)=a(x_n+\rho_ny)\rho_n^{-\gamma_2}.$ We denote by $\delta_n=d(x_n,\Gamma)=d(x_n,z_n),$ where $z_n\in\Gamma.$ We may also assume that $z_n\to z_0.$ Denote by $\nu_n$ the unit normal of $\Gamma$ at $z_n$ pointing inward of $\Omega^-.$ With this notation in hand it is easy to see
\[d_{\Gamma}(x_n+\rho_ny)=|\delta_n+\rho_n\nu_n.y+o(\rho_n)|\]
and 
\begin{equation}\label{eq4.42}
\left\{
\begin{aligned}
(a)~a_n(y)&=-\alpha_2(x_n+\rho_n y)|(\delta_n/\rho_n)+\nu_n.y+o(1)|^{\gamma_2}~~~\text{provided}~~x_n+\rho_n y\in\Omega^-,\\
(b)~a_n(y)&=\alpha_1(x_n+\rho_n y)|(\delta_n/\rho_n)+\nu_n.y+o(1)|^{\gamma_2}~~~\text{provided}~~x_n+\rho_n y\in\Omega^+.
\end{aligned}
\right.
\end{equation}
(i) Suppose that $\left\{\frac{\delta_n}{\rho_n}\right\}$ is bounded from above. By passing to a suitable subsequence, we may assume that $\frac{\delta_n}{\rho_n}\to s\ge0~~\text{as}~~n\to\infty.$ Then by \eqref{eq4.42}
\[a_n(y)\to h(s+\nu_0.y)~~\text{as}~~n\to\infty,\]   
where $h(t)=\alpha_1(x_0)(t^-)^{\gamma_2}-\alpha_2(x_0)(t^+)^{\gamma_2}.$
Then by applying Proposition \ref{lem:compactness}, we get
\begin{equation*}
\left\{
\begin{aligned}
&\lim\limits_{n\to\infty}v_n=v,~~\text{uniformly in}~~B_R(0)~\forall ~R>0,\\
&\text{and}\\
&v(0)=1.
\end{aligned}
\right.
\end{equation*}
Thus, in view of \eqref{eq4.41} and stability results results for viscosity solutions we have:
\begin{equation*}
\left\{
\begin{aligned}
-\mathcal{M}^{+}_{\lambda,\Lambda}(D^2v)&=h(s+\nu_0.y)v^p~~\text{in}~~\mathbb{R}^N,\\
v&\ge0~~\text{in}~~\mathbb{R}^N\\
v(0)&=\max_{\mathbb{R}^N}{v}=1.
\end{aligned}
\right.
\end{equation*}
As Pucci's operator is invariant under rotation and translation. Therefore, by appropriate translation and rotation of the coordinates we may assume $s+\nu_0.y=-y_1$ and hence $v$ is a solution of
\begin{equation*}
\left\{
\begin{aligned}
-\mathcal{M}^{+}_{\lambda,\Lambda}(D^2v)&=h(-y_1)v^p~~\text{in}~~\mathbb{R}^N,\\
v&\ge0~~\text{in}~~\mathbb{R}^N\\
\max_{\mathbb{R}^N}{v}&=1.
\end{aligned}
\right.
\end{equation*}
This contradicts  \cref{thm:weighted-liouville}.\\
(ii). Suppose that $\left\{\frac{\delta_n}{\rho_n}\right\}$ is unbounded from above. By selecting a suitable subsequence, we may assume  $\beta_n:=(\rho_n/\delta_n)^{{\gamma_2}/2}\to0$ as $n\to\infty.$ Let us set $z:=\beta_{n}^{-1}y$ and $w_n(z):=v_n(\beta_nz).$ Then by \cref{scaled:solution2}, $w_{n}$ satisfies 
\[-\mathcal{M}^{+}_{\lambda,\Lambda}(D^2w_n)=(\beta_n\rho_n)^2\mu_nw_n+\beta_n^2a_n(\beta_nz)w_n^p,\]
 Observe that
\[\beta_n^2a_n(\beta_nz)=-\alpha_2(x_n+\beta_n\rho_nz)[1+\beta_n^{1+2/\gamma_2}\nu_n.z+o(\beta_n)]^{\gamma_2}\to-\alpha_2(x_0)~~\text{as}~~n\to\infty.\]
It follows that the problem 
\begin{equation}\label{eq4.44}
\left\{
\begin{aligned}
-\mathcal{M}^{+}_{\lambda,\Lambda}(D^2w)&=-\alpha_2(x_0)w^p~~\text{in}~~\mathbb{R}^N,\\
w&\ge0~~\text{in}~~\mathbb{R}^N\\
w(0)&=1.
\end{aligned}
\right.
\end{equation}
has a solution.\\
Then by strong maximum principle, $w\equiv 1.$
Now, Take a constant test function $\phi(x)=1$, since $w(x)\leq w(0)=1$. Therefore, $\phi$ touches $w$ from above at $x=0.$ Since $w$ is a viscosity subsolution of \cref{eq4.44}, Then 
\begin{equation*}
\begin{aligned}
     -\mathcal{M}^{+}_{\lambda,\Lambda}(D^2\phi(0))&\leq-\alpha_2(x_0)w(0)^p\\
     0\leq -\alpha_2(x_0)
\end{aligned}   
\end{equation*}
which contradicts positivity of $\alpha.$

If $\{x_n\}\subset\overline{\Omega^+}$ occurs, then the proof for case (i) is similar but for case (ii), instead of \eqref{eq4.44}, we obtain
\begin{equation}\label{eq4.44+}
\left\{
\begin{aligned}
-\mathcal{M}^{+}_{\lambda,\Lambda}(D^2w)&=\alpha_1(x_0)w^p~~\text{in}~~\mathbb{R}^N,\\
w&\ge0~~\text{in}~~\mathbb{R}^N\\
w(0)&=1
\end{aligned}
\right.
\end{equation}
 which contradicting  Theorem 4.1 in  \cite{CutriLeoni2000}, the only viscosity supersolution of \eqref{eq4.44+} is $w\equiv0.$  \\
 \end{proof}
 
\begin{lemma}
Let $\{x_n\}$ be the sequence defined in Lemma \eqref{lem:2.5} and $u_n(x_n)\to\infty$ as $n\to\infty.$ Assume that $p\in(1,\psafe]~~\text{if}~~\NtildeP>2$ and $1<p<\infty~~\text{if}~~\NtildeP\le2$, where $\psafe$ is defined by \eqref{eq:safe-exponent}. Suppose further that
\[
a(x)=\alpha_3(x)\,[d(x,\partial Z_{+})]^{\gamma_3}
\quad \text{for } x\in\Omega^{+} \text{ near } \partial Z_{+},
\]
where \(\alpha_3(x)\) is a positive continuous function defined in a neighbourhood of \(\partial Z_{+}\), and \(\gamma_3>0\). Then \(\partial Z_{+}\) contains no limit point of the sequence \(\{x_n\}\).                                      
\end{lemma}
\begin{proof}
  The same argument as in Lemma \eqref{lemma3.3} can be applied.
\end{proof} 
\begin{lemma}
Let $\{x_n\}$ be the sequence defined in Lemma \eqref{lem:2.5} and $u_n(x_n)\to\infty $ as $n\to\infty.$ Assume that $p>1$ and assumption \ref{ass:geometry} (G2) holds. Suppose also
\[m\leq \frac{-a(x)}{[d(x,\Gamma_{-})]^{\gamma_4}} \leq M~~\text{for}~~x\in\Omega^{-}~~\text{near}~~\Gamma_{-},\]
where $0<m<M$ and $\gamma_4>0.$ Then $\Gamma_{-}$ does not contain any limit point of $\{x_n\}.$ Similarly, if the above condition on $a(x)$ holds for $x\in\Omega^{-}$ near $\partial Z_{-},$ then $\partial Z_{-}$ does not contain any limit point  contains no limit point of $\{x_n\}.$
\end{lemma}
\begin{proof}
Again we prove this results by method of contradiction. Suppose some subsequence of $\{x_{n}\}$ converges to $x_{0}\in\Gamma_{-}.$ In view of Lemma \eqref{lem:2.5} we may assume $\{x_n\}\subset \overline{\Omega^{-}}.$ Let us set 
\[M_n:=u_n(x_n),~\rho_n:=M_n^\frac{-(p-1)}{(2+\gamma_4)}\] 
and observe that $M_n\to\infty$ and $\rho_n\to0$ as $n\to\infty.$ Again as in the proof of the previous results define 
\[v_{n}(y)=\frac{u_n(x_{n}+\rho_{n}y)}{M_n}~~\text{for}~~y\in\Omega_{n}.\]
Also note that $v_{n}$ solves 
\begin{equation}\label{eq4.61}
-\mathcal{M}^{+}_{\lambda,\Lambda}(D^2v_n)=\rho_n^2\mu_nv_n+a_n(y)v_n^p~~\text{in}~~\Omega_{n},
\end{equation}
where $a_n(y)=a(x_n+\rho_ny)\rho_n^{-\gamma_4}.$ Let $\delta_n=d(x_n,\Gamma_-)=d(x_n,z_n),$ where $z_n\in\Gamma_-.$ We may assume $z_n\to z_0.$ Denote by $\nu_n$ the unit normal of $\Gamma_-$ at $z_n$ pointing inward to $\Omega^-.$ With this notation in hand we can write
\[d_{\Gamma_-}(x_n+\rho_ny)=|\delta_n+\rho_n\nu_n.y+o(\rho_n)|,\]
for large $n$ and all $y\in\Omega^-_{n}$ we have 
\begin{equation}\label{eq4.62}
-a_n(y)/|(\delta_n/\rho_n)+\nu_n.y+o(1)|^{\gamma_4}\in[m,M].
    \end{equation}
\begin{itemize}
\item[(i)] Suppose that $\left\{\frac{\delta_n}{\rho_n}\right\}$ is bounded from above. Then we can choose a subsequence, which we still denote by same sequence, such that $\frac{\delta_n}{\rho_n}\to s\ge0~~\text{as}~~n\to\infty.$ Thus in view of  \eqref{eq4.62}, we get
\[a_n(y)\leq-m[(\delta_n/\rho_n)+\nu_n.y+o(1)]^{\gamma_4}\to h(s+\nu_0.y)\] 
as $n\to\infty,$ where $h(t)=-m(t^+)^{\gamma_4}.$\\
Then by applying Proposition \ref{lem:compactness}, we obtain
\begin{equation*}
\left\{
\begin{aligned}
&\lim\limits_{n\to\infty}v_n=v,~~\text{uniformly in}~~B_R(0)~\forall ~R>0,\\
&\text{and}\\
&v(0)=1.
\end{aligned}
\right.
\end{equation*}
By taking \eqref{eq4.61},\eqref{eq4.62} into account and stability results for viscosity solution we get
\begin{equation*}
\left\{
\begin{aligned}
-\mathcal{M}^{+}_{\lambda,\Lambda}(D^2v)&\leq h(s+\nu_0.y)v^p~~\text{in}~~\mathbb{R}^N,\\
v&\ge0~~\text{in}~~\mathbb{R}^N\\
v(0)&=\max_{\mathbb{R}^N}{v}=1.
\end{aligned}
\right.
\end{equation*}
Again in view translation and rotation invariance property of the Pucci operator, we can translate and rotate the coordinates to get $s+\nu_0.y=-y_1.$ Thus $v$ solves
\begin{equation*}
\left\{
\begin{aligned}
-\mathcal{M}^{+}_{\lambda,\Lambda}(D^2v)&\leq h(-y_1)v^p\leq0~~\text{in}~~\mathbb{R}^N,\\
v&\ge0~~\text{in}~~\mathbb{R}^N\\
v(0)&=\max_{\mathbb{R}^N}{v}=1.
\end{aligned}
\right.
\end{equation*}
By the strong maximum principle, $v$ is constant, i.e., $v\equiv1 $. But this contradicts Theorem 4.1\cite{CutriLeoni2000}.\\
\item[(ii)] Now suppose $\left\{\frac{\delta_n}{\rho_n}\right\}$ is unbounded from above. By choosing a suitable subsequence, we may assume $\beta_n:=(\rho_n/\delta_n)^{{\gamma_4}/2}\to0$ as $n\to\infty.$ Let us denote $z=\beta^{-1}_{n}y$
and $w_n(z):=v_n(\beta_nz).$ With the help of this notation equation \eqref{eq4.61} can be rewritten as follows:
\begin{equation*}
-\mathcal{M}^{+}_{\lambda,\Lambda}(D^2w_n)=(\beta_n\rho_n)^2\mu_nw_n+\beta_n^2a_n(\beta_nz)w_n^p.
\end{equation*}
By using the stability result for viscosity solution we find 
\begin{equation}\label{eq4.64}
\left\{
\begin{aligned}
-\mathcal{M}^{+}_{\lambda,\Lambda}(D^2w)&\leq-mw^p\leq0~~\mathbb{R}^N,\\
w&\ge0~~\text{in}~~\mathbb{R}^N\\
w(0)&=1,
\end{aligned}
\right.
\end{equation}
where we have used 
\[\beta_n^2a_n(\beta_nz)\leq-m[1+\beta_n^{1+2/\gamma_4}\nu_n.z+o(\beta_n)]^{\gamma_4}\to-m~~\text{as}~~n\to\infty.\]
By the strong maximum principle, $w$ is constant, i.e., $w\equiv1 $. Now, Take a constant test function $\phi(x)=1$, since $w(x)\leq w(0)=1$. Therefore, $\phi$ touches $w$ from above at $x=0.$ Since $w$ is a viscosity subsolution of \cref{eq4.64}, Then 
\begin{equation*}
\begin{aligned}
     -\mathcal{M}^{+}_{\lambda,\Lambda}(D^2\phi(0))&\leq-mw(0)^p\\
     0\leq -m
\end{aligned}   
\end{equation*}
which contradicts positivity of $m.$
\end{itemize}    
\end{proof}
\begin{lemma}\label{lem:negative-interior-exclusion}
 Let $K_1,\ldots,K_m\Subset\Omega$ denote the thick neutral
components, so that
\[
K_j\subset\{a=0\},
\qquad j=1,\ldots,m.
\]
Assume that
\[
B<\min_{1\le j\le m}\mu_1^+(K_j^\circ).
\]

Let $\{u_n\}$ be a sequence of positive viscosity solutions of
\begin{equation}
\label{eq:forced-sequence}
-\mathcal M_{\lambda,\Lambda}^+(D^2u_n)
=
\mu_nu_n+a(x)u_n^p
\quad\text{in }\Omega,
\end{equation}
with
\[
u_n=0\quad\text{on }\partial\Omega,
\]
where
\[
\mu_n\le B.
\]

Let $\psi$ be the positive barrier constructed as in
Lemma \ref{lem:positive-barrier}, and choose $x_n\in\Omega$ such that
\[
\frac{u_n(x_n)}{\psi(x_n)}
=
\max_{\overline\Omega}\frac{u_n}{\psi}.
\]
Set
\[
M_n:=u_n(x_n).
\]

Suppose that
\[
M_n\to\infty
\qquad\text{and}\qquad
x_n\to x_0\in\Omega^-.
\]
Then such a sequence cannot exist.
Similarly, if $a(x)<0$ on $\partial\Omega$ Then $x_n$ does not have limiting point in $\partial\Omega.$
\end{lemma}

\begin{proof}
Since $x_0\in\Omega^-$ and $\Omega^-=\{a<0\}$ is open, there exists
$r>0$ such that
\begin{equation}
\label{eq:negative-ball}
\overline{B(x_0,r)}\subset\Omega^-.
\end{equation}

Since every thick neutral component satisfies
\[
K_j\subset\{a=0\},
\]
we have
\[
K_j\cap\overline{B(x_0,r)}=\varnothing,
\qquad j=1,\ldots,m.
\]
Both $K_j$ and $\overline{B(x_0,r)}$ are compact; hence
\[
d_j:=
\operatorname{dist}
\bigl(K_j,\overline{B(x_0,r)}\bigr)>0.
\]

By the spectral continuity of the principal half-eigenvalue under
smooth exterior approximation, the domains $D_j$ in the construction
of Lemma~\ref{lem:positive-barrier} may be chosen sufficiently close
to $K_j$ so that, simultaneously,
\begin{equation}
\label{eq:Dj-choice-negative}
K_j\Subset D_j\Subset\Omega,
\qquad
\mu_1^+(D_j)>B,
\qquad
D_j\cap\overline{B(x_0,r)}=\varnothing.
\end{equation}
Indeed, since
\[
\mu_1^+(D_j)\longrightarrow\mu_1^+(K_j^\circ)
\]
as the smooth exterior domains $D_j$ shrink to $K_j^\circ$, and
\[
B<\mu_1^+(K_j^\circ),
\]
the spectral inequality $\mu_1^+(D_j)>B$ is preserved when $D_j$ is
taken sufficiently close to $K_j$.

Let $\phi_j>0$ be the principal eigenfunction in $D_j$ satisfying
\[
-\mathcal M_{\lambda,\Lambda}^+(D^2\phi_j)
=
\mu_1^+(D_j)\phi_j
\quad\text{in }D_j,
\]
and choose
\[
\chi_j\in C_c^\infty(D_j),
\qquad
0\le\chi_j\le1,
\]
with $\chi_j\equiv1$ in a neighborhood of $K_j$.
For a fixed constant $c>0$, define
\begin{equation}
\label{eq:psi-negative}
\psi
=
\sum_{j=1}^m\chi_j\phi_j
+
c\left(1-\sum_{j=1}^m\chi_j\right).
\end{equation}

Since
\[
\operatorname{supp}\chi_j\subset D_j
\]
and, by \eqref{eq:Dj-choice-negative},
\[
D_j\cap B(x_0,r)=\varnothing,
\]
we have
\[
\chi_j\equiv0
\qquad\text{in }B(x_0,r)
\]
for every $j$. Consequently,
\begin{equation}
\label{eq:psi-constant-negative}
\psi\equiv c
\qquad\text{in }B(x_0,r),
\end{equation}
and hence
\begin{equation}
\label{eq:hessian-zero}
D^2\psi\equiv0
\qquad\text{in }B(x_0,r).
\end{equation}

Since $x_n\to x_0$, for all sufficiently large $n$,
\[
x_n\in B(x_0,r).
\]
At the maximum point of $u_n/\psi$, define
\[
s_n:=\frac{M_n}{\psi(x_n)}.
\]
Then $s_n\psi$ touches $u_n$ from above at $x_n$. Since $u_n$ is a
viscosity subsolution of \eqref{eq:forced-sequence}, we obtain
\[
s_n
\left[
-\mathcal M_{\lambda,\Lambda}^+
(D^2\psi(x_n))
\right]
\le
\mu_nM_n+a(x_n)M_n^p.
\]
Dividing by $M_n>0$ and using
\[
s_n=\frac{M_n}{\psi(x_n)},
\]
gives
\begin{equation}
\label{eq:key-negative-interior}
\frac{
-\mathcal M_{\lambda,\Lambda}^+(D^2\psi(x_n))
}{
\psi(x_n)
}
\le
\mu_n+a(x_n)M_n^{p-1}.
\end{equation}

By \eqref{eq:hessian-zero},
\[
-\mathcal M_{\lambda,\Lambda}^+(D^2\psi(x_n))=0
\]
for all sufficiently large $n$. Moreover,
\[
B(x_0,r)\subset\Omega^-,
\]
for all sufficiently large $n$. Therefore
\eqref{eq:key-negative-interior} reduces to
\begin{equation}
\label{eq:negative-contradiction}
0
\le
\mu_n+a(x_n)M_n^{p-1}.
\end{equation}

Since $a(x_0)<0$ and $a$ is continuous, there exist
$\delta>0$ and $n_0\in\mathbb N$ such that
\[
a(x_n)\le-\delta
\qquad\text{for all }n\ge n_0.
\]
Using also $\mu_n\le B$, we deduce from
\eqref{eq:negative-contradiction} that
\[
0
\le
B-\delta M_n^{p-1}.
\]
Since $p>1$ and $M_n\to\infty$,
\[
B-\delta M_n^{p-1}\longrightarrow-\infty,
\]
which is impossible.

Hence the ratio-maximizing points $x_n$ cannot converge to a point
$x_0\in\Omega^-$.\\
\textbf{Case (2):} If $x_n\to x_0\in\partial\Omega$ with $a(x)<0$ on $\partial\Omega.$ Then by definition of $\psi$, it is constant near $\partial\Omega.$ Then again by \cref{eq:key-negative-interior}, we get a contradiction. Therefore $x_0\notin\partial\Omega$ as well.
\end{proof}

\begin{lemma}
Let $\{x_n\}$ be the sequence defined in Lemma \ref{lem:2.5} and $u_n(x_n)\to\infty $ as $n\to\infty.$ Assume that $p\in(1,\psafe]~~\text{if}~~\NtildeP>2$ and $1<p<\infty~~\text{if}~~\NtildeP\le2$, where $\psafe$ is defined by \eqref{eq:safe-exponent}. If \(K\) is an arbitrary compact subset of \(\Omega^{+}\), then \(K\) contains no limit point of the sequence \(\{x_n\}\).
\end{lemma}
\begin{proof}
Let $K$ be an arbitrary compact subset of $\Omega^{+}$ and $x_{0}\in K,$ where $x_{0}$ is a limit point of $\{x_{n}\}.$ As in the proof of the previous results, without loss of generality we can assume that $x_{n}\to x_{0}.$ Also setting  $M_n:=u_n(x_n),$ $\rho_n:=M_n^\frac{-(p-1)}{2}$ we find $M_n\to\infty$ and $\rho_n\to0$ as $n\to\infty.$ 
Define 
\[v_n(y):=\frac{u_n(x_{n}+\rho_{n}y)}{M_n}~~\text{for}~~y\in\Omega_{n}.\]
By using the fact that $u_{n}$ solves \eqref{eq:problem1} with $\mu=\mu_{n},$ we find $v_{n}$ satisfies
\begin{equation}\label{change2}
-\mathcal{M}^{+}_{\lambda,\Lambda}(D^2v_n)=\rho_n^2\mu_nv_n+a_n(y)v_n^p~~\text{in}~~\Omega_{n}
\end{equation}
where $a_n(y)=a(x_n+\rho_ny)$ for $x_n+\rho_n y\in\Omega^{+}.$ Given any $R>0,$ it follows from $\rho_n\to0$ and $K\subset\subset\Omega^{+}$ that there exists $n_R$ such that $x_n+\rho_n y\in\Omega^{+}$ for each $n\geq n_R$ and all $y\in\mathbb{R}^N$ with $|y|\leq R+1.$ We also observe that 
\[0<v_n(y)\leq\frac{\psi(x_n+\rho_n y)}{\psi(x_n)}
\leq
\frac{\max_{\overline{\Omega}}\psi}
     {\min_{\overline{\Omega}}\psi}=C_\psi,~~~|y|\leq R+1,~~\text{for}~~n\geq n_R.\]
     where $\psi$ is as in Lemma \ref{lem:2.5}.\\
Then by applying Proposition \ref{lem:compactness}, we get 
\begin{equation*}
\left\{
\begin{aligned}
&\lim\limits_{n\to\infty}v_n=v~~\text{uniformly in}~~B_R(0)~\forall~R>0\\
&\text{and}\\
&v(0)=1
\end{aligned}
\right.
\end{equation*}
By using the above convergence and stability result for viscosity solution we find
\begin{equation*}
\left\{
\begin{aligned}
-\mathcal{M}^{+}_{\lambda,\Lambda}(D^2v)&=a(x_0)v^p ~~~\text{in}~~\mathbb{R}^N\\
v(0)&=1
\end{aligned}
\right.
\end{equation*}
where $a(x_0)>0,$ which contradicts Theorem 4.1 in \cite{CutriLeoni2000}.
\end{proof}
\begin{lemma}\label{lem3.11}
Let $\{x_n\}$ be the sequence defined in Lemma \ref{lem:2.5} and $u_n(x_n)\to\infty $ as $n\to\infty.$ Assume that $p\in(1,\psafe]~~\text{if}~~\NtildeP>2$ and $1<p<\infty~~\text{if}~~\NtildeP\le2$, where $\psafe$ is given by \eqref{eq:safe-exponent}. Let $I\subset\mathbb R$ be a bounded interval with $\mu_n\in I$.Let $\psi$ be the fixed positive barrier from Lemma \ref{lem:2.5},
Suppose further that $a(x)>0$ on $\partial\Omega.$ Then $\partial\Omega$ contains no limit point of the sequence $\{x_n\}.$
\end{lemma}
\begin{proof}
We adapt the boundary blow-up argument of
\cite[Proposition~12]{charro2009multiplicity}.

Suppose, to the contrary, that, after passing
to a subsequence,
\[
x_n\longrightarrow x_0\in\partial\Omega,
\qquad
M_n\longrightarrow+\infty.
\]
Write
\[
a_0:=a(x_0)>0,
\qquad
C_\mu:=\sup_n|\mu_n|<\infty.
\]
The maximality of $u_n/\psi$ gives
\begin{equation}\label{eq:lemma514-ratio}
\frac{u_n(x)}{M_n}
\le
\frac{\psi(x)}{\psi(x_n)}
\qquad (x\in\overline{\Omega}).
\end{equation}

Define
\[
\rho_n:=M_n^{-(p-1)/2},
\qquad
\Omega_n:=\frac{\Omega-x_n}{\rho_n},
\qquad
v_n(y):=\frac{u_n(x_n+\rho_n y)}{M_n}.
\]
Then $\rho_n\to0$ and
$\rho_n^2M_n^{p-1}=1$.
Positive homogeneity of the Pucci operator
yields, in the viscosity sense,
\begin{equation}\label{eq:lemma514-scaled}
\begin{cases}
-\mathcal{M}_{\lambda,\Lambda}^{+}(D^2v_n)
=
\rho_n^2\mu_n v_n
+a(x_n+\rho_n y)v_n^p
& \text{in }\Omega_n,\\
v_n=0
& \text{on }\partial\Omega_n.
\end{cases}
\end{equation}
Indeed, if $\phi$ touches $v_n$ from above
or below at $y_*$, then
\[
\Phi(x):=
M_n\phi\left(\frac{x-x_n}{\rho_n}\right)
\]
touches $u_n$ in the same manner at
$x_n+\rho_n y_*$.
Since
\[
D^2\Phi
=
\frac{M_n}{\rho_n^2}D^2\phi,
\]
the corresponding viscosity inequality for
\eqref{eq:problem1}, divided by
$M_n/\rho_n^2$, is precisely the required
inequality for \eqref{eq:lemma514-scaled}.

By \eqref{eq:lemma514-ratio},
\begin{equation}\label{eq:lemma514-bound}
0\le v_n(y)
\le
\frac{\psi(x_n+\rho_n y)}{\psi(x_n)}
\le C_\psi,
\qquad
v_n(0)=1,
\end{equation}
where
\[
C_\psi:=
\frac{\max_{\overline{\Omega}}\psi}{\min_{\overline{\Omega}}\psi}.
\]
Uniform continuity of $\psi$ and its positive
minimum imply that
\begin{equation}\label{eq:lemma514-ratio-limit}
\frac{\psi(x_n+\rho_n y)}{\psi(x_n)}
\longrightarrow1
\end{equation}
uniformly on every fixed bounded set of
admissible points $y$.
Moreover,
\begin{equation}\label{eq:lemma514-parameter}
\rho_n^2\mu_n\longrightarrow0.
\end{equation}
For large $n$, the right-hand sides of
\eqref{eq:lemma514-scaled} have the uniform bound
\[
\left|
\rho_n^2\mu_n v_n
+a(x_n+\rho_n y)v_n^p
\right|
\le
C_\mu C_\psi
+\|a\|_{L^\infty(\Omega)}C_\psi^p.
\]

Set
\[
\delta_n:=
\operatorname{dist}(x_n,\partial\Omega),
\qquad
d_n:=\frac{\delta_n}{\rho_n}.
\]
Although $\delta_n\to0$, the quotient $d_n$
need not tend to zero.
Passing to a further subsequence, there are
three possibilities.

\medskip
\noindent
\textit{Case 1: $d_n\to+\infty$.}

For every $R>0$, eventually
$\delta_n>R\rho_n$, and hence
$B_R(0)\subset\Omega_n$.
The interior H\"older estimates,
\eqref{eq:lemma514-bound}, and a diagonal
extraction provide a locally uniform limit
$v$ on $\mathbb{R}^N$.
By \eqref{eq:lemma514-bound} and
\eqref{eq:lemma514-ratio-limit},
\[
0\le v\le1
\quad\text{in }\mathbb{R}^N,
\qquad
v(0)=1.
\]
Since
$a(x_n+\rho_n y)\to a_0$
locally uniformly, viscosity stability and
\eqref{eq:lemma514-parameter} imply
\begin{equation}\label{eq:lemma514-entire}
-\mathcal{M}_{\lambda,\Lambda}^{+}(D^2v)=a_0v^p
\qquad\text{in }\mathbb{R}^N.
\end{equation}

The whole-space Liouville theorem forces
$v\equiv0$ in the stated exponent range,
a contradiction.
For $\widetilde N_+>2$, the precise
maximal-Pucci input is
\cite[Remark~8 following Theorem~4.1]{CutriLeoni2000}.
For $\widetilde N_+\le2$,
\cite[Theorem~3.2]{CutriLeoni2000}
makes every nonnegative entire supersolution
of $-\mathcal{M}_{\lambda,\Lambda}^{+}(D^2v)\ge0$ constant, and
\eqref{eq:lemma514-entire} makes this
constant zero.
The coefficient $a_0$ is inessential,
since $a_0^{1/(p-1)}v$ solves the equation
with coefficient one.
In dimension one, the same conclusion
follows directly from concavity.

\medskip
\noindent
\textit{Boundary coordinates for the remaining cases.}

Suppose that $d_n\to s\in[0,\infty)$.
The $C^2$ boundary has a tubular neighborhood
in which the nearest boundary point is unique.
Thus, for large $n$, there exist
$z_n\in\partial\Omega$ and inward unit normals
$\nu_n$ at $z_n$ such that
\[
x_n=z_n+\delta_n\nu_n,
\qquad
z_n\longrightarrow x_0.
\]
Choose orthogonal matrices $Q_n$ with
$Q_ne_N=\nu_n$, and define
\[
\widehat\Omega_n
:=
\{y\in\mathbb{R}^N:
z_n+\rho_nQ_ny\in\Omega\},
\qquad
w_n(y):=
\frac{u_n(z_n+\rho_nQ_ny)}{M_n}.
\]
Orthogonal invariance and positive homogeneity
yield
\begin{equation}\label{eq:lemma514-boundary-scaled}
\begin{cases}
-\mathcal{M}_{\lambda,\Lambda}^{+}(D^2w_n)
=
\rho_n^2\mu_nw_n
+a(z_n+\rho_nQ_ny)w_n^p
& \text{in }\widehat\Omega_n,\\
w_n=0
& \text{on }\partial\widehat\Omega_n.
\end{cases}
\end{equation}
Furthermore,
\begin{equation}\label{eq:lemma514-boundary-values}
0\le w_n\le C_\psi,
\qquad
w_n(0)=0,
\qquad
w_n(d_ne_N)=1.
\end{equation}

In the coordinates centered at $z_n$ and
rotated by $Q_n$, the boundary is locally a
graph $\xi_N=g_n(\xi')$, with the interior
above it, where
\[
g_n(0)=0,
\qquad
Dg_n(0)=0,
\qquad
\|D^2g_n\|_\infty\le C
\]
on a fixed neighborhood, uniformly in $n$.
The rescaled graph is
\[
y_N=G_n(y'),
\qquad
G_n(y'):=
\rho_n^{-1}g_n(\rho_n y').
\]
For every fixed $R>0$, Taylor's theorem gives
\[
\|G_n\|_{C^2(B_R')}
\le C_R\rho_n\longrightarrow0.
\]
Consequently, $\widehat\Omega_n$ converges
locally, with uniformly controlled boundary
charts, to the half-space
\[
H:=\{y_N>0\}.
\]

\medskip
\noindent
\textit{Case 2: $d_n\to0$.}

The uniform bound in
\eqref{eq:lemma514-boundary-values},
the uniform bound on the right-hand side of
\eqref{eq:lemma514-boundary-scaled},
and the uniformly controlled boundary charts
imply a uniform boundary H\"older estimate;
see \cite[Theorem~2.3]{QuaasSirakov2006}.
Thus, there exist $C>0$ and
$\alpha\in(0,1)$, independent of $n$, such that
\[
|w_n(y)-w_n(0)|
\le C|y|^\alpha
\qquad
\left(
y\in
\overline{\widehat\Omega_n}\cap B_1(0)
\right).
\]
For large $n$, $d_ne_N\in B_1(0)$, so
\eqref{eq:lemma514-boundary-values} gives
\[
1
=
|w_n(d_ne_N)-w_n(0)|
\le Cd_n^\alpha
\longrightarrow0,
\]
a contradiction.

\medskip
\noindent
\textit{Case 3: $d_n\to s\in(0,\infty)$.}

The interior and boundary H\"older estimates,
followed by a diagonal extraction, give a
limit $w\in C(\overline H)$, with local uniform
convergence in $H$ and convergence up to the
limiting boundary.
The zero boundary data pass to the limit.
Since $d_ne_N\to se_N\in H$,
\eqref{eq:lemma514-boundary-values} gives
\[
w(se_N)=1.
\]
By \eqref{eq:lemma514-ratio},
\[
w_n(y)
\le
\frac{\psi(z_n+\rho_nQ_ny)}{\psi(x_n)}.
\]
Both arguments of $\psi$ tend to $x_0$,
uniformly on fixed bounded sets.
Hence $0\le w\le1$ in $H$.
Finally, continuity of $a$,
\eqref{eq:lemma514-parameter},
and viscosity stability yield
\begin{equation}\label{eq:lemma514-half-space}
\begin{cases}
-\mathcal{M}_{\lambda,\Lambda}^{+}(D^2w)=a_0w^p
& \text{in }H,\\
0\le w\le1
& \text{in }H,\\
w=0
& \text{on }\partial H,\\
w(se_N)=1.
\end{cases}
\end{equation}

We now apply \cite[Theorem~4]{charro2009multiplicity}
with
\[
G(X):=-\mathcal M_{\lambda,\Lambda}^{+}(X),
\qquad f(t):=a_0t^p.
\]
The operator $G$ is uniformly elliptic, positively
homogeneous, and invariant under orthogonal
transformations. Moreover, since $p>1$ and $a_0>0$,
the function $f$ is locally Lipschitz on $[0,\infty)$
and satisfies $f(0)=0$.

The function $w$ is a bounded, nontrivial,
nonnegative viscosity solution of
\[
G(D^2w)=f(w)\quad\text{in }\mathbb R^N_+,
\qquad
w=0\quad\text{on }\partial\mathbb R^N_+.
\]
Therefore, Theorem~4 implies that $w$ is strictly
increasing in the $y_N$-direction. Consequently,
for every $t>0$,
\[
w((s+t)e_N)>w(se_N)=1,
\]
contradicting $w\le1$ in $\mathbb R^N_+$.
This excludes the case $d_n\to s\in(0,\infty)$.\\
All possible behaviors of $d_n$ have been
excluded.
Thus $\partial\Omega$ contains no limit
point of $\{x_n\}$.
\end{proof}
We are now in a position to prove ~\cref{thm:apriori}.

\begin{proof}[Proof of \cref{thm:apriori}]
Arguing by contradiction, suppose that there exists a sequence of positive
solutions $\{u_n\}$ of \cref{eq:problem1}, with $\mu=\mu_n\in\Ipar$, such that
\[
\|u_n\|_\infty\to\infty.
\]

By Proposition \ref{prop:necessary-spectral-restriction}, we have $\mu_n<\sigma_*$. Moreover, by
Lemma \ref{lem:2.5}, the maximum points $x_n$ of $\frac{u_n}{\psi}$ with $u_n(x_n):=M_n\to\infty$ (by \cref{max:inequality}) may be chosen so that
\[
    x_n\in\overline{\Omega^-}\cup\overline{\Omega^+}.
\]
Since $\overline{\Omega}$ is compact, after passing to a subsequence, we may
assume that
\[
    x_n\longrightarrow x_0\in\overline{\Omega}.
\]
As $\overline{\Omega^-}\cup\overline{\Omega^+}$ is closed, it follows that
\[
    x_0\in\overline{\Omega^-}\cup\overline{\Omega^+}.
\]
On the other hand, Lemma \ref{lemma3.3} to \ref{lem3.11} exclude all possible limit
points of $\{x_n\}$ in
$\overline{\Omega^-}\cup\overline{\Omega^+}$. Hence,
\[
    x_0\notin\overline{\Omega^-}\cup\overline{\Omega^+},
\]
which is a contradiction. Therefore, the desired uniform a priori estimate
follows.

\end{proof}
\subsection{Existence of positive solution}
\begin{proposition}
\label{prop:positive-bifurcation}
Let $\Omega\subset\mathbb{R}^{N}$ be a bounded domain with $C^{2}$
boundary, let \cref{eq:problem1} holds
Then $\bigl(\mu_{1}^{+}(\Omega),0\bigr)$ is a bifurcation
point of positive viscosity solutions of
\eqref{eq:problem1}.
\end{proposition}
\begin{proof}
 Define
\[
\mathcal{N}:C_0(\overline{\Omega})
\longrightarrow C(\overline{\Omega}),
\qquad
\mathcal{N}(u)(x)=a(x)(u^{+}(x))^{p}.
\]
Since $a\in C(\overline{\Omega})$ and $p>1$,
$\mathcal{N}$ is continuous. Moreover,
\[
\frac{|\mathcal{N}(u)|}
     {|u|}
\leq
|a|
|u|^{p-1}
\longrightarrow0
\]
as $|u|\to0$.
Thus
\[
\mathcal{N}(u)=o(|u|)
\quad\text{in }C(\overline{\Omega}).
\]   
Then, by Theorem 1.3 \cite{MR2124162}, $\bigl(\mu_{1}^{+}(\Omega),0\bigr)$ is a bifurcation
point of positive  solutions to
\eqref{eq:problem1}. Moreover, there exists a closed connected set
$\mathcal{C}^{+}\subset\mathbb{R}\times C_{0}(\overline{\Omega})$
such that
\begin{enumerate}
\item
\[
\bigl(\mu_{1}^{+}(\Omega),0\bigr)
 \in\mathcal{C}^{+};
\]

\item every
\[
(\mu,u)\in
\mathcal{C}^{+}\setminus
\bigl\{(\mu_{1}^{+}(\Omega),0)\bigr\}
\]
satisfies
\[
u>0\quad\text{in }\Omega;
\]

\item either $\mathcal{C}^{+}$ is unbounded in
$\mathbb{R}\times C_{0}(\overline{\Omega})$, or it contains
a point $(\widehat{\mu},0)$, where
$\widehat{\mu}\neq\mu_{1}^{+}(\Omega)$ is an eigenvalue of
$-\mathcal{M}_{\lambda,\Lambda}^{+}$.
\end{enumerate}
\end{proof}
\begin{theorem}
\label{thm:existence-full-interval}

Set $\kappa=\mu_1^+(\Omega)$ and
\[
X=\{u\in C(\overline\Omega):u=0\text{ on }\partial\Omega\},
\qquad
\mathcal S_+=\{(\mu,u)\in\mathbb R\times X:
 u\text{ solves \eqref{eq:problem1}},\ u>0\text{ in }\Omega\}.
\]
All solutions are understood in the viscosity sense.

Assume that the positive global continuum in Proposition \ref{prop:positive-bifurcation} is
available in the following precise form: there exists a nontrivial
closed connected set
$\mathcal C^+\subset\overline{\mathcal S_+}$ containing $(\kappa,0)$,
which is either unbounded in $\mathbb R\times X$ or contains a point
$(\widehat\mu,0)$ with $\widehat\mu\ne\kappa$.
The closure is taken in $\mathbb R\times X$.
Suppose also that:
\begin{enumerate}
\item[(P1)] There exists $\sigma\in\mathbb R$ such that
 $\mu<\sigma$ for every $(\mu,u)\in\mathcal S_+$.
\item[(P2)] For every bounded closed interval $I\subset\mathbb R$
 there exists $C_I>0$ such that
 \[
 \|u\|_{L^\infty(\Omega)}\le C_I
 \quad\text{for every }(\mu,u)\in\mathcal S_+
 \text{ with }\mu\in I.
 \]
\end{enumerate}
Then, for every $\mu<\mu_1^+(\Omega)$,
problem \eqref{eq:problem1} has a positive viscosity solution.
\end{theorem}

\begin{proof}
We first identify the trivial points in $\overline{\mathcal S_+}$.
Suppose $(\widehat\mu,0)\in\overline{\mathcal S_+}$.
Choose $(\mu_n,u_n)\in\mathcal S_+$ with
$\mu_n\to\widehat\mu$ and
$\varepsilon_n:=\|u_n\|_\infty\to0$, and put
$v_n=u_n/\varepsilon_n$.
Positive homogeneity gives
\[
-\mathcal M^+_{\lambda,\Lambda}(D^2v_n)
 =\mu_n v_n+a(x)\varepsilon_n^{p-1}v_n^p
 \quad\text{in }\Omega,
\qquad v_n=0\quad\text{on }\partial\Omega,
\]
with $0\le v_n\le1$ and $\|v_n\|_\infty=1$.
The right-hand sides are uniformly bounded.
The interior and zero-boundary H\"older estimates yield, after passing
to a subsequence, uniform convergence on $\overline\Omega$ to a function
$v\ge0$ with $\|v\|_\infty=1$.
Since $p>1$, by stability of viscosity solutions
\[
-\mathcal M^+_{\lambda,\Lambda}(D^2v)=\widehat\mu v
 \quad\text{in }\Omega,
\qquad v=0\quad\text{on }\partial\Omega.
\]
The strong minimum principle implies $v>0$ in $\Omega$.
The characterization of the positive principal half-eigenvalue
therefore gives $\widehat\mu=\kappa$.

Every nontrivial point of $\overline{\mathcal S_+}$ is a positive
solution, by viscosity stability and the strong minimum principle.
The preceding normalization excludes every trivial point of
$\mathcal C^+$ other than $(\kappa,0)$.
Its assumed global alternative therefore implies that
$\mathcal C^+$ is unbounded.

Suppose that its parameter projection were bounded below by some
$m\in\mathbb R$.
Set $B=\max\{\sigma,\kappa\}$.
By (P1), every point of $\mathcal C^+$ has parameter in the bounded
closed interval $[m,B]$.
By (P2), the amplitudes of all its nontrivial points are bounded
by $C_{[m,B]}$, while its only trivial point has amplitude zero.
Thus $\mathcal C^+$ is bounded in $\mathbb R\times X$, a contradiction.
Consequently,
\[
\inf\pi_\mu(\mathcal C^+)=-\infty,
\qquad \pi_\mu(\mu,u)=\mu.
\]
Since $\mathcal C^+$ is connected and contains $(\kappa,0)$,
its continuous parameter projection is an interval containing
$\kappa$ and unbounded below.
Hence $(-\infty,\kappa)\subset\pi_\mu(\mathcal C^+)$.
For each $\mu<\kappa$, a point $(\mu,u)\in\mathcal C^+$ cannot be
trivial, because its only trivial point is $(\kappa,0)$.
It is therefore a positive solution of
\eqref{eq:problem1}.
\end{proof}

\section*{Declarations}

\subsection*{Ethical Approval}
Not applicable.

\subsection*{Funding}
Mohan Mallick acknowledges the support of the Anusandhan National Research Foundation through grant ANRF/ARGM/2025/002309/MTR. Ram Baran Verma acknowledges the support of the Anusandhan National Research Foundation through grant  \\ ANRF/ARGM/2025/002357/MTR and the National Board for Higher Mathematics through grant 02011/36/2025/NBHM/RP/9466.

\subsection*{Availability of data and materials}
Not applicable. No datasets were generated or analysed during the current study.

\bibliographystyle{amsplain}
\bibliography{references_C1}

\end{document}